\documentclass[12pt]{amsart}

\usepackage[margin=1in]{geometry}
\usepackage{amsmath,amsthm,amssymb}
\usepackage{dsfont}
\usepackage{amsmath,amscd}
\usepackage{mathrsfs}
\usepackage{color}
\usepackage{hyperref}
\usepackage{enumitem}
\usepackage{array}
\usepackage{graphicx}
\usepackage{xcolor}
\usepackage{ulem}

\hypersetup{
    colorlinks=true, 
    linkcolor=blue, 
    urlcolor=red, 
    linktoc=all 
}

\input xy
\xyoption{all}

\newenvironment{example}{\noindent\textbf{Example:}\hspace{1em}}{}

\newcommand{\N}{\mathbb{N}}
\newcommand{\Q}{\mathbb{Q}}

\newcommand{\C}{\mathbb{C}}

\newcommand{\PP}{\mathbb{P}}
\newcommand{\Z}{\mathbb{Z}}

\newcommand{\one}{\mathds{1}}

\theoremstyle{plain}
\newtheorem{thm}{Theorem}[section]
\newtheorem{prop}[thm]{Proposition}
\newtheorem{lemma}[thm]{Lemma}
\newtheorem{remark}[thm]{Remark}
\newtheorem{notation}[thm]{Notation}
\newtheorem{cor}[thm]{Corollary}
\newtheorem{problem}[thm]{Problem}

\theoremstyle{definition}
\newtheorem{dfn}{Definition}[thm]

\begin{document}

\title{Realization of graded monomial ideal rings modulo torsion}
\author{Tseleung So}
\address{Department of Mathematics and Statistics, Univeristy of Regina, Regina, SK S4S 0A2, Canada}
\email{Tse.Leung.So@uregina.ca}
\author{Donald Stanley}
\address{Department of Mathematics and Statistics, Univeristy of Regina, Regina, SK S4S 0A2, Canada}
\email{Donald.Stanley@uregina.ca}


\maketitle

\begin{abstract}
Let $A$ be the quotient of a graded polynomial ring $\Z[x_1,\ldots,x_m]\otimes\Lambda[y_1,\ldots,y_n]$ by an ideal generated by monomials with leading coefficients 1. Then we constructed a space~$X_A$ such that $A$ is isomorphic to $H^*(X_A)$ modulo torsion elements.
\end{abstract}

\section{Introduction}

A classical problem in algebraic topology asks: which commutative graded $R$-algebra $A$ are isomorphic to $H^*(X_A;R)$ for some space $X_A$? The space $X_A$, if it exists, is called a realization of $A$. According to Aguad\'e \cite{aguade} the problem goes back to at least Hopf, and was later explicitly stated by Steenrod \cite{steenrod}.
To solve the problem in general is probably too ambitious, but many special cases have been proven.

One of Quillen's motivations for his seminal work on rational homotopy theory \cite{quillen} was to solve this problem over $\mathbb{Q}$. He showed that all simply connected graded $\mathbb{Q}$-algebras have a realization. The problem of which polynomial algebras over $\mathbb{Z}$ have realizations has a long history and a complete solution was given by Anderson and Grodal \cite{AG}, (see also \cite{notbohm}).  More recently Trevisan~\cite{trevisan} and later Bahri-Bendersky-Cohen-Gitler~\cite{BBCG2} constructed realizations of $\Z[x_1,\ldots,x_m]/I$, where $|x_i|=2$ and $I$ is an ideal generated by monomials with leading coefficient 1.

We want to consider a related problem that lies between the solved realization problem over $\mathbb{Q}$ and the very difficult realization problem over 
$\mathbb{Z}$. We do this by modding out torsion. 

\begin{problem}
Which commutative graded $R$-algebra $A$ are isomorphic to $H^*(X_A;R)/\text{torsion}$ for some space $X_A$?
\end{problem}

Such an $X_A$ is called a realization modulo torsion of $A$. For example, a polynomial ring~$\Z[x]$ has a realization modulo torsion given by Eilenberg-MacLane space $K(\Z,|x|)$ if $|x|$ is even, while $\Z[x]$ has a realization (before modding out torsion) if and only if $|x|=2$ or $4$~\cite{steenrod}. Here we ask: do all finite type connected commutative graded $\Z$-algebras have a realization modulo torsion?

Notice that modding out by torsion is different from taking rational coefficients. For example, both $H^*(\Omega S^{2n+1};\Q)$ and $H^*(K(\Z, 2n);\Q)$ are $\Q[x]$ generated by $x$ of degree $2n$. But~$H^*(K(\Z,2n))/\text{torsion}$ is $\Z[x]$, while $H^*(\Omega S^{2n+1})\cong\Gamma[x]$ is free as a $\Z$-module and is the divided polynomial algebra generated by $x$.

In this paper, we construct realizations modulo torsion of graded monomial ideal 
rings~$A$ which are tensors of polynomial algebras and exterior algebras modulo monomial ideals. More precisely, let $P=\Z[x_1,\ldots,x_m]\otimes\Lambda[y_1,\ldots,y_n]$ be a graded polynomial ring wher $x_i$'s have arbitrary positive even degrees and $y_j$'s have arbitrary positive odd degrees, and let $I=(M_1,\ldots,M_r)$ be an ideal generated by $r$ minimal monomials
\[
M_j=x_1^{a_{1j}}x_2^{a_{2j}}\cdots x_m^{a_{mj}}\otimes y_1^{b_{1j}}\cdots y_n^{b_{nj}},\quad 1\leq j\leq r,
\]
where indices $a_{ij}$ are non-negative integers and $b_{ij}$ are either 0 or 1. Then the quotient algebra $A=P/I$ is called a \emph{graded monomial ideal ring}.

\begin{thm}[Main Theorem]\label{thm_main thm}
Let $A$ be a graded monomial ideal ring. Then there exists a space $X_A$ such that $H^*(X_A)/T$ is isomorphic to $A$, where $T$ is the ideal consisting of torsion elements in $H^*(X_A)$. Moreover, there is a ring morphism $A\to H^*(X_A)$ that is right inverse to the quotient map $H^*(X_A)\to H^*(X_A)/T\cong A$.
\end{thm}

If all of the even degree generators are in degree $2$, then we do not need to mod out by torsion and so we get a generalization (Proposition~\ref{cor_deg 2 and odd deg}) of the results of Bahri-Bendersky-Cohen-Gitler~\cite[Theorem 2.2]{BBCG2} and Trevisan~\cite[Theorem 3.6]{trevisan}.

The structure of the paper is as follows. Section 2 contains preliminaries, algebraic tools and lemmas that are used in later sections. In Section 3 we recall the definition of polyhedral products and modify a result of Bahri-Bendersky-Cohen-Gitler~\cite{BBCG} to compute~$H^*((\underline{X},\ast)^{K})/T$. In Sections 4 and 5 we prove Theorem~\ref{thm_main thm} in several steps. First, we prove it in the special case where the ideal $I$ is square-free. Then for the general case, we construct a fibration sequence inspired by algebraic polarization method and show that the fiber $X_A$ is a realization modulo torsion of $A$. In Section 6 we illustrate how to construct~$X_A$ for an easy example of~$A$.

\section{Preliminaries}

\subsection{Quotients of algebras by torsion elements}\
It is natural to study an algebra $A$ by factoring out the torsion elements since the quotient algebra is torsion-free and has a simpler structure. Driven by this, we start investigating the quotients of cohomology rings of spaces by their torsion elements. Since we cannot find related references in the literature, here we fix the notation and develop lemmas for our purpose.

A graded module $A=\{A_i\}_{i\in S}$ is a family of indexed modules $A_i$. Since we are interested in cochain complexes and cohomology rings of connected, finite type CW-complexes, we assume $A$ to be a connected, finite type graded module with non-positive degrees. That is~$S=\N_{\leq0}$, $A_0=\Z$ and each component $A_i$ is finitely generated. We follow the convention and denote $A_{i}$ by $A^{-i}$.

\begin{remark}
Equivalently we can define a graded module to be a module with a grading structure, that is the direct sum $A=\bigoplus_{i\in S}A_i$ of a family of indexed modules. This definition is slightly different from the definition above. We will use both definitions interchangeably.
\end{remark}

An element $x\in A$ is \emph{torsion} if $cx=0$ for some non-zero integer $c$, and is \emph{torsion-free} otherwise. The torsion submodule $A_t$ of $A$ is the graded submodule consisting of torsion elements and the torsion-free quotient module $A_f=A/A_t$ is their quotient. If $B$ is another graded module and $g:A\to B$ is a morphism, then it induces a morphism $g_f:A_f\to B_f$ sending $a+A_t\in A_f$ to $g(a)+B_t\in B_f$. This kind of structure is important in abelian categories and was formalized with Dixon's notion of a torsion theory \cite{dixon}, but in this paper we only use the structure in a naive way. 

\begin{lemma}\label{lemma_free SES}
If $0\to A\overset{g}{\to}B\overset{h}{\to}C\to0$ is a short exact sequence of graded modules, then we have $C_f\cong(B_f/A_f)_f$. Furthermore, if the sequence is split exact, then so is
\[
0\longrightarrow A_f\overset{g_f}{\longrightarrow}B_f\overset{h_f}{\longrightarrow}C_f\longrightarrow0.
\]
\end{lemma}

\begin{proof}
Consider a commutative diagram
\[
\xymatrix{
	&0\ar[d]	&0\ar[d]	&0\ar[d]	&\\
0\ar[r]	&A_t\ar[d]\ar[r]^-{g_t}	&B_t\ar[d]\ar[r]^-{p}	&B_t/A_t\ar[d]^-{u}\ar[r]	&0\\
0\ar[r]	&A\ar[d]\ar[r]^-{g}		&B\ar[d]\ar[r]^-{h}		&C\ar[d]^-{v}\ar[r]			&0\\
0\ar[r]	&A_f\ar[d]\ar[r]^-{g_f}	&B_f\ar[d]\ar[r]^-{q}	&B_f/A_f\ar[d]\ar[r]	&0\\
	&0	&0	&0	&
}
\]
where $g_t$ is the restriction of $g$ to $A_t$, $p$ and $q$ are the quotient maps, and $u$ and $v$ are the induced maps. By construction all rows and columns are exact sequences except for the right column. A diagram chase implies that $u$ is injective and $v$ is surjective. We claim that the column is exact at $C$. Obviously $v\circ u$ is trivial. Take an element $c\in ker(v)$ and its preimage $b\in B$. A diagram chase implies $b=g(a)+b'$ for some $a\in A$ and $b'\in B_t$. So~$c=h(b')=u\circ p(b')$ is in $Im(u)$ and the right column $0\to B_t/A_t\overset{u}{\to}C\overset{v}{\to}B_f/A_f\to0$ is exact.

For the first part of the lemma, we show that $v_f:C_f\to (B_f/A_f)_f$ is an isomorphism. Since $v$ is surjective, so is $v_f$. Take $c'\in ker(v_f)$ and its preimage $\tilde{c}'\in C$. Then $v(\tilde{c}')$ is a torsion element in $B_f/A_f$ and $mv(\tilde{c}')=0$ for some non-zeor integer $m$. So $m\tilde{c}'\in{ker}(v)$. As~$ker(v)=Im(u)$ consists of torsion elements, $m\tilde{c}'$ is a torsion and so is $\tilde{c}'$. Therefore $c'=0$ in $C_f$ and $v_f$ is injective.

Notice that an exact sequence being split is equivalent to $B\cong A\oplus C$. So $B_f\cong A_f\oplus C_f$ and $0\to A_f\overset{g_f}{\to}B_f\overset{h_f}{\to}C_f\to0$ is a split exact sequence.
\end{proof}

A graded algebra $(A,m)$ consists of a graded module $A$ and an associative bilinear multiplication $m=\{m^{i,j}:A^i\otimes A^j\to A^{i+j}\}$ such that $1\in A^0$ is the multiplicative identity. A pair $(M,\mu)$ is a left $A$-module (or right $A$-module) if $M$ is a graded module and $\mu$ is an associative bilinear multiplication $\mu=\{\mu^{i,j}:A^i\otimes M^j\to M^{i+j}\}$ such that $\mu(1\otimes x)=x$ for all $x\in M$ (respectively $\mu=\{\mu^{i,j}:M^i\otimes A^j\to M^{i+j}\}$ such that $\mu(1,x)=x$). We check that modding out torsion and multiplications are compatible.

\begin{lemma}\label{lemma_tensor product free}
If $A$ and $M$ are graded modules (not necessarily of finite type), then there is a unique isomorphism $\theta:(A\otimes M)_f\to A_f\otimes M_f$ of graded modules making the diagram
\[
\xymatrix{
A\otimes M\ar[r]\ar[d]	&A_f\otimes M_f\\
(A\otimes M)_f\ar[ur]_-{\theta}
}
\]
commute, where the vertical and the horizontal maps are quotient maps.
\end{lemma}

\begin{proof}
It suffices to show that $(A^i\otimes M^j)_f\cong A^i_f\otimes M^j_f$ for any positive integers $i$ and $j$. Consider commutative diagram
\[
\xymatrix{
0\ar[r]	&A^i_t\otimes M^j\oplus A^i\otimes M^j_t\ar[r]^-{\imath_1}\ar[d]^-{a}	&A^i\otimes M^j\ar@{=}[d]\ar[r]^-{\pi_1}	&A^i_f\otimes M^j_f\ar[d]^-{b}\ar[r]	&0\\
0\ar[r]	&(A^i\otimes M^j)_t\ar[r]^-{\imath_2}	&A^i\otimes M^j\ar[r]^-{\pi_2}	&(A^i\otimes M^j)_f\ar[r]	&0\\
}
\]
where $a,\imath_1$ and $\imath_2$ are inclusions, $\pi_1$ and $\pi_2$ are quotient maps and $b$ is the induced map. We want to show that $b$ is an isomorphism, which is equivalent to showing that $a$ is an isomorphism. If $A$ and $M$ are of finite type, then $a$ is an isomorphism since $A^i$ and $M^j$ are finitely generated abelian groups. In the general case, $a$ is an isomorphism by~\cite[Theorem 61.5]{fuchs}.
\end{proof}

\begin{cor}\label{cor_torsion free and multiplication compatible}
Let $(A,m)$ be a graded algebra and let $m'_f$ be the composition
\[
m'_f:A_f\otimes A_f\cong(A\otimes A)_f\overset{m_f}{\to}A_f.
\]
Then $(A_f,m'_f)$ is a graded algebra and there is a commutative diagram
\[
\xymatrix{
A\otimes A\ar[r]^-{m}\ar[d]	&A\ar[d]\\
A_f\otimes A_f\ar[r]^-{m'_f}	&A_f
}
\]
where the vertical maps are quotient maps.

Let $(M,\mu)$ be a left (or right) $A$-module and let $\mu'_f$ be the composition
\[
\begin{array}{c c}
\mu'_f:A_f\otimes M_f\cong(A\otimes M)_f\overset{\mu_f}{\to}M_f.
&(\text{respectively }\mu'_f:M_f\otimes A_f\cong(M\otimes A)_f\overset{\mu_f}{\to}M_f)
\end{array}
\]
Then $(M_f,\mu'_f)$ is a left (respectively right) $A_f$-module and there is a commuative diagram
\[
\begin{array}{c c}
\vcenter{
\xymatrix{
A\otimes M\ar[r]^-{\mu}\ar[d]	&M\ar[d]\\
A_f\otimes M_f\ar[r]^-{\mu'_f}	&M_f
}}
&\left(\text{respectively }\vcenter{
\xymatrix{
M\otimes A\ar[r]^-{\mu}\ar[d]	&M\ar[d]\\
M_f\otimes A_f\ar[r]^-{\mu'_f}	&M_f
}}
\right)
\end{array}
\]
where the vertical maps are quotient maps.
\end{cor}

A \emph{cochain complex} $(A,d)$ consists of a graded module $A$ and a differential
\[
d=\{d^i:A^i\to A^{i+1}\}
\]
such that $d\circ d=0$. Let $d_f=\{d^i_f:A^i_f\to A^{i+1}_f\}$ be the induced differential on $A_f$. Then~$(A_f,d_f)$ forms a cochain complex and its cohomology $H^*(A_f,d_f)=\{H^i(A_f,d_f)\}_{i\geq0}$ is a graded module. 

A \emph{differential graded algebra} $(A,m,d)$ is a cochain complex $(A,d)$ such that $(A,m)$ is a graded algebra and $d$ and $m$ satisfy the Leibniz rule. Let $d_t$ be the restriction of $d$ to $A_t$. Then $(A_t,d_t)$ is a differential ideal and $(A_f,d_f)$ is a differential graded algebra, so $H^*(A_f,d_f)$ is a graded algebra.

A left (respectively right) \emph{dg-algebra module} $(M,\mu,\delta)$ over $(A,m,d)$ if $(M,\mu)$ is a left (respectively right) $(A,m)$-module, $(M,\delta)$ is a cochain complex and $\delta$ and $\mu$ satisfy the Leibniz rule. Then $H^*(M_f,\delta_f)$ is a left (respectively right) $H^*(A_f)$-module.

Even if $(A_f,d_f)$ is torsion-free, $H^*(A_f,d_f)$ is not necessarily torsion-free. Denote $(H^*(A,d))_f$ by $H^*_f(A,d)$. The following lemma compares $H^*_f(A,d)$ and $H^*_f(A_f,d_f)$.

\begin{lemma}\label{lemma_free cohmlgy}
Let $(A,d)$ be a cochain complex. Then there is a monomorphism of modules
\[
\psi:H^*_f(A,d)\longrightarrow H^*_f(A_f,d_f).
\]
If $H^{i+1}(A_t,d_t)=0$, then $\psi:H^i_f(A,d)\to H^i_f(A_f,d_f)$ is an isomorphism. Moreover, suppose~$(A,m,d)$ is a differential graded algebra. Then $\psi$ is a morphism of algebras.
\end{lemma}

\begin{proof}
Assume $(A,d)$ is a cochain complex. Let $\imath:(A_t,d_t)\to(A,d)$ be the inclusion and let $\pi:(A,d)\to(A_f,d_f)$ be the quotient map. Then the short exact sequence of cochain complexes $
0\to(A_t,d_t)\overset{\imath}{\to}(A,d)\overset{\pi}{\to}(A_f,d_f)\to0$ induces a long exact sequence
\[
\cdots\to H^{i-1}(A_f, d_f)\to H^i(A_t, d_t)\overset{\imath^*}{\to}H^i(A,d)\overset{\pi^*}{\to}H^i(A_f,d_f)\to H^{i+1}(A_t, d_t)\to\cdots
\]
Take $\psi:H^*_f(A,d)\to H^*_f(A_f,d_f)$ to be the morphism induced by $\pi^*:H^*(A,d)\to H^*(A_f,d_f)$. We show that it has the asserted properties.

To show the injectivity of $\psi$, take an equivalence class $[a]\in H^*_f(A,d)$ such that $\psi[a]=0$. Represent it by a cocycle class $a\in H^i(A,d)$. Then $\pi^*(a)$ is torsion and $\pi^*(ca)=0$ for some non-zero number $c$. By exactness $ca\in Im(\imath^*)$. Since $H^i(A_t,d_t)$ is torsion, so is $Im(\imath^*)$ and~$ca$ is a torsion. Therefore $a\in H^i(A,d)$ is a torsion. By definition, $[a]\in H^i_f(A,d)$ is zero. So $\psi$ is injective.

Suppose $A^{i+1}$ has no torsion elements. Then $A^{i+1}_t=0$ and $H^{i+1}(A_t,d_t)=0$. So $\pi^*$ is surjective. By definition we have commutative diagram
\[
\xymatrix{
H^i(A,d)\ar[r]^-{\pi^*}\ar[d]	&H^i(A_f,d_f)\ar[d]\\
H^i_f(A,d)\ar[r]^-{\psi}		&H^i_f(A_f,d_f),
}\]
where vertical arrows are quotient maps and are surjective. So $\psi:H^i_f(A,d)\to H^i_f(A_f,d_f)$ is surjective and hence isomorphic.

If $A$ is a differential graded algebra, then $\pi^*:H^*(A,d)\to H^*(A_f,d_f)$ is a morphism of graded algebras. By Corollary~\ref{cor_torsion free and multiplication compatible} the induced morphism $\psi$ is multiplicative.
\end{proof}

\begin{example}
The surjectivity of $\psi:H^i_f(A,d)\to H^i_f(A_f,d_f)$ may fail if $A^{i+1}$ contains torsion elements. Let $(A,d)$ be a cochain complex where
\[
A^i=\begin{cases}
\Z		&i=0\\
\Z/2\Z	&i=1\\
0		&\text{otherwise},
\end{cases}
\]
and $d^i$ are trivial for all $i$ except for $d^0:\Z\to\Z/2\Z$ being the quotient map. Then $H^0(A)$ and $H^0(A_f)$ are $\Z$ while $\psi:H^0(A)\to H^0_f(A)$ is multiplication $2:\Z\to\Z$.
\end{example}

\subsection{Eilenberg-Moore spectral sequence}\label{section_EMSS}
Given a differential graded algebra $(A,d)$ and a right $A$-module $(M,d_M)$, first we define the bar bicomplex $\text{B}^{*,*}(M,A)$ as follows. For any positive integer $i$, let $\text{B}^{-i}(M,A)=M\otimes(\bar{A})^{\otimes i}$ where $\bar{A}=\{A^n\}_{n>0}$. Denote an element in~$\text{B}^{-i}(M,A)$ by $x[a_1|\cdots|a_i]$ for $x\in M$ and $a_i\in\bar{A}$. Let $\text{B}^{-i,j}(M,A)$ be the submodule of~$\text{B}^{-i}(M,A)$ consisting elements $x[a_1|\cdots|a_i]$ such that $|x|+\sum^i_{k=1}|a_k|=j$. The internal and external differentials
\[
\begin{array}{c c c}
d_I:\text{B}^{-i,j}(M,A)\to\text{B}^{-i,j+1}(M,A)
&\text{and}
&d_E:\text{B}^{-i,j}(M,A)\to\text{B}^{-i+1,j}(M,A)
\end{array}
\]
are given by
\begin{eqnarray*}
d_I(x[a_1|\cdots|a_i])
&=&(d_Mx)[a_1|\cdots|a_i]+\sum^{i}_{j=1}(-1)^{\epsilon_{j-1}}x[a_1|\cdots|a_{j-1}|d_Aa_j|a_{j+1}|\cdots|a_i]\\
d_E(x[a_1|\cdots|a_i])
&=&(-1)^{|x|}(xa_1)[a_2|\cdots|a_i]
+\sum^{i-1}_{j=1}(-1)^{\epsilon_{j}}x[a_1|\cdots|a_{j-1}|a_{j}\cdot a_{j+1}|\cdots|a_i],
\end{eqnarray*}
where $\epsilon_k=k+|x|+\sum^k_{j=1}|a_j|$. Then we define the bar construction $(\mathcal{B}(M,A),d_{\mathcal{B}})$ to be a graded module where $\mathcal{B}(M,A)^n=\displaystyle\bigoplus_{-i+j=n}\text{B}^{-i,j}(M,A)$ and $d_{\mathcal{B}}=\displaystyle\bigoplus_{-i+j=n}(d_I+d_E)$ for $n\geq 0$.

Take the filtration $\mathscr{F}^{-p}=\bigoplus_{0\leq i\leq p}\text{B}^{-i}(M,A)$. The associated spectral sequence $\{E^{*,*}_r\}^{\infty}_{r=0}$ is the Eilenberg-Moore spectral sequence converging to $H^*(\mathcal{B}(M,A))$~(see \cite[Remark 2.3]{FHT} and~\cite[Corollary 7.9]{mccleary}).


\begin{lemma}\label{lemma_free E2_alg}
Let $A$ be a simply connected differential graded algebra and $M$ be a right~$A$-module such that $A$ and $M$ are free as $\Z$-modules. Then there is a monomorphism of modules
\[
\psi:(E^{-p,q}_2)_f\longrightarrow\left(\text{Tor}^{-p,q}_{H_f(A)}(H_f(M),\Z)\right)_f.
\]
which is an isomorphism for $p=0$. Moreover, if $H(A)$ and $H(M)$ are free modules, then~$E^{-p,q}_2\cong\text{Tor}^{-p,q}_{H(A)}(H(M),\Z)$.
\end{lemma}

\begin{proof}
The $E_0$-page is given by
\[
E^{-p,*}_0=\mathscr{F}^{-p}/\mathscr{F}^{-p+1}=M\otimes(\bar{A}^{\otimes p})
\]
and $d_0=d_I$. By K\"{u}nneth Theorem the $E_1$-page is given by
\begin{eqnarray*}
E^{-p,*}_1
&\cong&H(M)\otimes(\tilde{H}(A)^{\otimes p})\oplus T\\
&\cong&\text{B}^{-p}(H(M),H(A))\oplus T
\end{eqnarray*}
where $T$ is a torsion term, and $d_1$ is induced by $d_E$. Denote $H(M)$ by $M'$ and $H(A)$ by $A'$ for short. By Lemma~\ref{lemma_tensor product free}, there is an isomorphism of graded modules
\[
\theta:(E_1^{-p,*})_f\cong(\text{B}^{-p}(M',A'))_f\to\text{B}^{-p}(M'_f,A'_f)
\]
such that
\[
\xymatrix{
\text{B}^{-p}(M',A')\ar[d]\ar[dr]	&\\
(\text{B}^{-p}(M',A'))_f\ar[r]^-{\theta}	&\text{B}^{-p}(M'_f,A'_f)
}
\]
where the downward maps are quotient maps. Let $d'$ be the external differential of $\text{B}^*(M'_f,A'_f)$. Then $\theta:((\text{B}^{-p}(M',A'))_f,(d_1)_f)\to(\text{B}^*(M'_f,A'_f),d')$ is an isomorphism of cochain complexes. By Lemma~\ref{lemma_free cohmlgy} there is a monomorphism of graded modules
\[
\psi:(E^{-p,q}_2)_f=H^{-p}_f(E_1^{*,q},d_1)
\to H^{-p}_f((\text{B}^{*,q}(M',A'))_f,(d_1)_f)\cong H^{-p}_f(\text{B}^{*,q}(M'_f,A'_f),d').
\]

Notice that $\text{B}^*(M'_f,A'_f)\cong M'_f\otimes_{A'_f}\text{B}^*(A'_f,A'_f)$ and $d'=\one\otimes_{A'_f}d''$, where $d''$ is the external differential of $\text{B}^{*}(A'_f,A'_f)$. Since~\cite[Proposition 7.8]{mccleary}
\[
\cdots\longrightarrow \text{B}^{-1}(A'_f,A'_f)\overset{d''}{\longrightarrow}\text{B}^{0}(A'_f,A'_f)\overset{\epsilon}{\longrightarrow}\Z\longrightarrow0
\]
is a projective resolution of $\Z$ over $A'_f$-modules where $\epsilon:\text{B}^{0}(A'_f,A'_f)\cong A'_f\to\Z$ is the augmentation, the monomorphism becomes
\[
\psi:(E^{-p,q}_2)_f\longrightarrow\left(\text{Tor}^{-p,q}_{A'_f}(M'_f,\Z)\right)_f.
\]
Since $\text{B}^{1}(M',A')=0$, $\psi$ is isomorphic for $p=0$ by Lemma~\ref{lemma_free cohmlgy}.

Suppose $H(A)$ and $H(M)$ are free $\Z$-modules. By the K\"{u}nneth Theorem, we have
\[
E^{*,*}_1\cong\text{B}^{*,*}(H(M),H(A))
\]
and $d_1$ is the external differential. So $E^{-p,q}_2\cong\text{Tor}^{-p,q}_{H(A)}(H(M),\Z)$.
\end{proof}

Let $F\to E\overset{\pi}{\to}X$ be a fibration sequence where all spaces are connected, finite type CW-complexes, and $X$ is simply connected. In~\cite[Theorem III]{FHT} there is a quasi-isomorphism
\[
\Theta:\Omega(C^{\pi}_*(E),C_*(X))\to CN_*(F)
\]
of dg-algebra modules, which is natural in $\pi$. Here $\Omega(-,-)$ is the cobar construction, $C^{\pi}_*(E)$ is a non-negative chain complex, $C_*(X)$ is a simply connected chain complex, $CN_*(F)$ is a chain complex and $C^{\pi}_*(E)$, $C_*(X)$ and $CN_*(F)$ are quasi-isomorphic to the singular chain complexes of $E,X,F$ respectively.

Denote the dual of a (co)chain complex $C$ by $C^{\vee}=\text{Hom}(C,\Z)$. Since $B$ is simply connected, $H^1(X)=0$ and $H^2(X)$ is free. By~\cite[Propositions 4.2 and 4.6]{FHT} there are finite type graded free modules $V=\{V^{i}\}_{i\geq2}$ and $W=\{W^j\}_{j\geq0}$, a quasi-isomorphism of dg-algebras
\[
\phi:T(V)\to(C_*(X))^{\vee}
\]
and a quasi-isomorphism of dg-algebra modules
\[
\varphi:T(V)\otimes W\to(C^{\pi}_*(E))^{\vee},
\]
where $T(V)$ is the tensor algebra on $V$. Write $\tilde{X}=T(V)$ and $\tilde{E}=T(V)\otimes W$ for short. Then the compositions
\[
\begin{array}{c c c}
C_*(X)\overset{\text{incl.}}{\to}(C^*(X))^{\vee}\overset{\phi^{\vee}}{\to}\tilde{X}^{\vee}
&\text{and}
&C^{\pi}_*(E)\overset{\text{incl.}}{\to}(C^*(E))^{\vee}\overset{\varphi^{\vee}}{\to}\tilde{E}^{\vee}
\end{array}
\]
are quasi-isomorphisms of dg-coalgebras and of dg-coalgebra modules. Since $C_*(X)$ and~$\tilde{X}^{\vee}$ are simply connected free chain complexes and $C^{\pi}_*(E)$ and $\tilde{E}^{\vee}$ are non-negative chain complexes, we have a zig-zag of quasi-isomorphisms
\[
\Omega(\tilde{E}^{\vee},\tilde{X}^{\vee})\overset{\simeq}{\longleftarrow}\Omega(C^{\pi}_*(E),C_*(X))\overset{\Theta}{\longrightarrow}CN_*(F).
\]
Since $\tilde{E}$ and $\tilde{X}$ are of finite type, dualize the zig-zag and take cohomology to get an isomorphism
\[
H^*(\mathcal{B}(\tilde{E}, \tilde{X}))\overset{\cong}{\longrightarrow}H^*(F).
\]
The Eilenberg-Moore spectral sequence $\{E^{*,*}_r\}^{\infty}_{r=0}$ on $F\to E\overset{\pi}{\to}X$ is the Eilenberg-Moore spectral sequence given by $A=\tilde{X}$ and $M=\tilde{E}$. Note that this definition depends on the choice of the pair $(\tilde{X},\tilde{E},\phi,\varphi)$. Any two choices may give spectral sequences with different~$E_0$-pages, but their $E_r$-pages are isomorphic for $r\geq 1$.


\begin{lemma}\label{lemma_free E2}
Let $F{\to}E\overset{\pi}{\to}X$ be a fibration sequence such that all spaces are finite type spaces and $X$ is simply connected, and let $\{E^{-p,q}_2\}$ be the $E_2$-page of Eilenberg-Moore spectral sequence on this fibration. Then there is a monomorphism $\psi:(E^{-p,q}_2)_f\to\left(\text{Tor}^{-p,q}_{H^*_f(X)}(H^*_f(E),\Z)\right)_f$ as modules such that $\psi$ is an isomorphism for $p=0$.
\end{lemma}

\begin{proof}
Since $H(\tilde{E})\cong H^*(E)$ and $H(\tilde{X})\cong H^*(X)$, Lemma~\ref{lemma_free E2_alg} implies that there is a monomorphism $\psi:(E^{-p,q}_2)_f\to(\text{Tor}^{-p,q}_{H^*_f(X)}(H^*_f(E),\Z))_f$ such that $\psi$ is an isomorphism at~$p=0$.
\end{proof}

Recall that the $E_0$-page is given by $E^{p,*}_0=\mathscr{F}^{-p}/\mathscr{F}^{-p+1}\cong\tilde{E}\otimes(\overline{\tilde{X}})^{\otimes p}$. In particular, if~$p=0$, then $E^{0,*}_0\cong\tilde{E}$. On the other hand, $\{E^{*,*}_r\}^{\infty}_{r=0}$ is a second quadrant spectral sequence. So $E^{0,*}_r$ is the kernel of the differential map and $E^{0,*}_{r+1}$ is a quotient group of $E^{0,*}_r$. For $r\in\N\cup\{\infty\}$, define the edge homomorphism $e_r$ to be the composition
\[
e_r:H(E)\cong H(\tilde{E})\cong E^{0,*}_1\to E^{0,*}_r
\]
where the unnamed arrow is the quotient map. The following lemma tells how the edge homomorphisms relate the $E_r$-page to $H^*(E)$ and $H^*(F)$.

\begin{lemma}\label{lemma_edge homomorphism}
Under the hypotheses of Lemma~\ref{lemma_free E2}, the edge homomorphisms make the diagram
\[
\xymatrix{
H^*(E)\ar@{=}[r]\ar[d]^-{e_1}	&H^*(E)\ar@{=}[r]\ar[d]^-{e_2}	&\cdots\ar[r]	&H^*(E)\ar@{=}[r]\ar[d]^-{e_{\infty}}	&H^*(E)\ar[d]^-{\imath^*}\\
E^{0,*}_1\ar[r]^-{\jmath_1}	&E^{0,*}_2\ar[r]^-{\jmath_2}	&\cdots\ar[r]	&E^{0,*}_{\infty}\ar[r]^-{\jmath}	&H^*(F)
}
\]
commute, where $\imath^*$ is induced by $\imath:F\to E$, $\jmath$ is the inclusion and $\jmath_r$'s are the quotient maps.
\end{lemma}

\begin{proof}
We use the notation above. Consider commutative diagram
\[
\xymatrix{
F\ar[r]^-{\imath}\ar[d]^-{\imath}	&E\ar[r]^-{\pi}\ar@{=}[d]	&X\ar[d]\\
E\ar@{=}[r]							&E\ar[r]^-{c}				&\text{pt}
}
\]
where $c$ is the constant map. Since the quasi-isomorphism $\Theta$ is natural, we have
\[
\xymatrix{
\Omega(C^{\pi}_*(E),C_*(X))\ar[r]^-{\Theta}\ar[d]	&CN_*(F)\ar[d]^-{\imath_*}\\
\Omega(C^{c}_*(E),C_*(\text{pt}))\ar[r]^-{\Theta}	&CN_*(E)
}
\]
The supplement $\Z\to(C_*(\text{pt}))^{\vee}$ is a quasi-isomorphism of dg-algebras and $\varphi:\tilde{E}\to(C^{\pi}_*(E))^{\vee}$ is a quasi-isomorphism of dg-algebra modules. Using this replacement and taking dual and cohomology of the diagram, we obtain
\begin{equation}\label{diagram_edge homo}
\xymatrix{
H^*(\tilde{E})\ar[r]^-{\cong}\ar[d]^-{e^*}	&H^*(E)\ar[d]^-{\imath^*}\\
H^*(\mathcal{B}(\tilde{E} ,\tilde{X}))\ar[r]^-{\cong}	&H^*(F)
}
\end{equation}
Here $e^*$ is the composition
\[
e^*:H^*(\tilde{E})\cong H^*(\mathcal{B}(\tilde{E},\Z))\overset{e'}{\to}H^*(\mathcal{B}(\tilde{E},\tilde{X}))
\]
and $e'$ is induced by the inclusion $e:\text{B}^{*,*}(\tilde{E},\Z)\to\text{B}^{*,*}(\tilde{E},\tilde{X})$. Let $\{\hat{E}^{*,*}_{r}\}^{\infty}_{r=0}$ be the Eilenberg-Moore spectral sequence on $E\overset{=}{\to}E\overset{c}{\to}\text{pt}$. Then $\hat{E}^{*,*}_0\cong\text{B}^{*,*}(\tilde{E},\Z)$ and the~$\hat{E}_1$-page collapses to $H^*(\tilde{E})$. The inclusion $e:\text{B}^{*,*}(\tilde{E},\Z)\to\text{B}^{*,*}(\tilde{E},\tilde{X})$ gives the following commutative diagram
\[
\xymatrix{
H^*(\tilde{E})\ar@{=}[r]\ar[d]^-{\tilde{e}_1}	&H^*(\tilde{E})\ar@{=}[r]\ar[d]^-{\tilde{e}_2}	&\cdots\ar[r]	&H^*(\tilde{E})\ar@{=}[r]\ar[d]^-{\tilde{e}_{\infty}}	&H^*(\tilde{E})\ar[d]^-{e^*}\\
E^{0,*}_1\ar[r]^-{\jmath_1}	&E^{0,*}_2\ar[r]^-{\jmath_2}	&\cdots\ar[r]	&E^{0,*}_{\infty}\ar[r]^-{\tilde{\jmath}}	&H^*(\mathcal{B}(\tilde{E},\tilde{X}))
}
\]
where $\tilde{e}_r:H^*(\tilde{E})\cong H^*(E)\overset{e_r}{\to}E^{0,*}_r$ and $\tilde{\jmath}:E^{0,*}_{\infty}\overset{\jmath}{\to}H^*(F)\cong H^*(\mathcal{B}(\tilde{E},\tilde{X}))$. Combine this to~(\ref{diagram_edge homo}) and obtain the asserted commutative diagram.
\end{proof}

\subsection{Regular sequences and freeness}

Here we use the alternative description of graded objects. A commutative graded algebra $A=\bigoplus_{i\geq0}A_i$ is an algebra with a grading such that~$ab=(-1)^{ij}ba$ for $a\in A^i$ and $b\in A^j$, and a graded~$A$-module $M=\bigoplus_{j\geq0}M_j$ is the direct sum of a family of $A$-modules. A set~$\{r_1,\ldots,r_n\}$ of elements in $M$ is called an~\emph{$M$-regular sequence} if ideal $(r_1,\ldots,r_n)M$ is not equal to $M$ and the multiplication
\[
r_i:M/(r_1,\ldots,r_{i-1})M\to M/(r_1,\ldots,r_{i-1})M
\]
is injective for $1\leq i\leq n$. In the special case where $M$ is a $\mathbb{K}[x_1,\ldots,x_n]$-module for some field $\mathbb{K}$ and the grading of $M$ has a lower bound, $M$ is a free $\mathbb{K}[x_1,\ldots,x_n]$-module if~$\{x_i\}^n_{i=1}$ is a regular sequence in $M$. We want to extend this fact to the case where $M$ is a~$\Z[x_1,\ldots,x_n]$-module. Recall a corollary of the graded Nakayama Lemma.

\begin{lemma}\label{lemma_nakayama}
Let $A$ be a graded ring and let $M$ be an $A$-module. Suppose $A$ and $M$ are non-negatively graded, and $I=(r_1,\ldots,r_n)\subset A$ is an ideal generated by homogeneous elements~$r_i$ of positive degrees. If $\{m_{\alpha}\}_{\alpha\in S}$ is a set of homogeneous elements in $M$ whose images generate~$M/IM$, then $\{m_{\alpha}\}_{\alpha\in S}$ generates $M$.
\end{lemma}

\begin{lemma}\label{lemma_regular freeness}
Let $M$ be a $\Z[x_1,\ldots,x_n]$-module with non-negative degrees. If $M/(x_1,\ldots,x_n)M$ is a free $\Z$-module and $\{x_1,\ldots,x_n\}$ is an $M$-regular sequence, then $M$ is a free $\Z[x_1,\ldots,x_n]$-module.
\end{lemma}

\begin{proof}
Denote $I=(x_1,\ldots,x_n)$. By assumption there is a set $\{m_{\alpha}\}_{\alpha\in S}$ of homogeneous elements in $M$ such that their quotient images form a basis in $M/IM$. By Lemma~\ref{lemma_nakayama} $\{m_{\alpha}\}_{\alpha\in S}$ generates $M$. We need to show that $\{m_{\alpha}\}_{\alpha\in S}$ is linear independent over $\Z[x_1,\ldots,x_n]$.

For $0\leq i\leq n$, let $M_i=M/(x_1,\ldots,x_{n-i})M$, $A_i=\Z[x_{n+1-i},\ldots,x_n]$ and $m_{\alpha,i}$ be the quotient image of $m_{\alpha}$ in $M_i$. Prove that $M_i$ is a free $A_i$-module with a basis $\{m_{\alpha,i}\}_{\alpha\in S}$ by induction on $i$. For $i=0$, $M_0=M/IM$ and $A_0=\Z$. The statement is true since $\{m_{\alpha,0}\}_{\alpha\in S}$ is a basis by construction. Assume the statement holds for $i\leq k$. For $i=k+1$, if there is a collection $\{f_\alpha\}_{\alpha\in S}$ of polynomials satisfying
\begin{equation}\label{eqn_linear combination}
\sum_{\alpha\in S}f_{\alpha}\cdot m_{\alpha,k+1}=0,
\end{equation}
we show that all $f_{\alpha}$'s are zero.

If not, then there are finitely many non-zero polynomials $f_{j_1},\ldots,f_{j_r}$. Quotient $M_{k+1}$ and~$A_{k+1}$ by ideal $(x_{n-k})$ and let $\bar{f}_{j_i}$ be the image of $f_{j_i}$ in $A_k$. Then~(\ref{eqn_linear combination}) becomes
\[
\sum^r_{i=1}\bar{f}_{j_i}\cdot m_{j_i,k}=0.
\]
By inductive assumption, $\{m_{\alpha,k}\}$ is a basis in $M_k$. So $\bar{f}_{j_i}=0$ and $f_{j_i}=x_{n-k}g_{j_i}$ for some polynomial $g_{j_i}\in A_{k+1}$. Since $x_{n-k}$ is not a zero-divisor, putting $f_{j_i}=x_{n-k}g_{j_i}$ to~(\ref{eqn_linear combination}) and get
\[
\sum^r_{i=1}g_{j_i}\cdot m_{j_i,k+1}=0.
\]
So $g_{j_1},\ldots,g_{j_r}$ are non-zero polynomials satisfying~(\ref{eqn_linear combination}) and $|g_{j_i}|<|f_{j_i}|$ for $1\leq i\leq r$. Iterating this argument implies that $|f_{j_i}|$'s are arbitrarily large, but this is impossible. So $f_{j_i}$'s must be zero and $\{m_{\alpha,k+1}\}$ is linear independent. It follows that $M_{k+1}$ is a free $A_{k+1}$-module. 
\end{proof}

\section{Cohomology rings of polyhedral products}
Let $[m]=\{1,\ldots,m\}$, let $ K$ be a simplicial complex on $[m]$ and let $(\underline{X},\underline{A})=\{(X_i,A_i)\}^m_{i=1}$ be a sequence of pairs of relative CW-complexes. For any simplex $\sigma\in K$ define
\[
(\underline{X}, \underline{A})^{\sigma}=\left\{(x_1,\ldots,x_m)\in\left.\prod^m_{i=1}X_i\right|x_i\in A_i\text{ for }i\notin\sigma\right\}
\]
as a subspace of $\prod^m_{i=1}X_i$, and define the \emph{polyhedral product}
\[
(\underline{X},\underline{A})^{K}=\bigcup_{\sigma\in K}(\underline{X},A)^{\sigma}
\]
to be the union of $(\underline{X},\underline{A})^{\sigma}$ over $\sigma\in K$.

If $X_i=\C\PP^{\infty}$ and $A_i=\ast$ for all $i$, then $(\C\PP^{\infty},\ast)^{K}$ is homotopy equivalent to Davis-Januszkiewicz space~\cite[Theorem 4.3.2]{BP}. For any principal ideal domain $R$, $H^*((\C\PP^{\infty},\ast)^{K};R)$ is isomorphic to the Stanley-Reisner ring $R[x_1,\ldots,x_m]/I_{K}$. Here $I_{K}$ is the ideal generated by $x_{j_1}\cdots x_{j_k}$ for $x_{j_i}\in\tilde{H}^*(X_{j_i};R)$ and $\{j_1,\ldots,j_k\}\notin K$, and is called the \emph{Stanley-Reisner ideal of $K$}. In general, a similar formula holds for $H^*((\underline{X},\ast)^K)$ whenever $X_i$'s are any spaces with free cohomology.

\begin{thm}[\cite{BBCG}]\label{prop_BBCG}
Let $R$ be a principal ideal domain, let $ K$ be a simplicial complex on $[m]$ and let $\underline{X}=\{X_i\}^m_{i=1}$ be a sequence of CW-complexes. If $H^*(X_i;R)$ is a free $R$-module for all $i$, then
\[
H^*((\underline{X},\ast)^{K};R)\cong\bigotimes^m_{i=1}H^*(X_i;R)/I_{K},
\]
where $I_{K}$ is generated by $x_{j_1}\otimes\cdots\otimes x_{j_k}$ for $x_{j_i}\in\tilde{H}^*(X_{j_i};R)$ and $\{j_1,\ldots,j_k\}\notin K$ and is called the generalized Stanley-Reisner ideal of $K$.
\end{thm}

The proof of Theorem~\ref{prop_BBCG} uses the strong form of K\"{u}nneth Theorem, which says that
\[
\mu:\bigotimes^m_{i=1}H^*(X_i;R)\to H^*(\prod^m_{i=1}X_i;R),\quad x_1\otimes\cdots\otimes x_m\mapsto\pi^*_1(x_1)\cup\cdots\cup\pi^*_m(x_m),
\]
where $\pi^*_j$ is induced by projection $\pi_j:\prod^m_{i=1}X_i\to X_j$, is isomorphic if all $H^*(X_i;R)$'s are free. In the reduced version of K\"{u}nneth Theorem, $\bar{\mu}:\bigotimes^m_{i=1}\tilde{H}^*(X_i)\to\tilde{H}^*(\bigwedge^m_{i=1}X_i)$ is also isomorphic if all $\tilde{H}^*(X_i;R)$'s are free. The goal of this section is to modify Theorem~\ref{prop_BBCG} by removing the freeness assumption on $H^*(X_i)$'s. As a trade-off, we need to mod out the torsion elements of $H^*(X_i)$'s. First let us refine K\"{u}nneth Theorem.

\begin{lemma}\label{lemma_free Kunneth}
Let $\underline{X}=\{X_i\}^m_{i=1}$ be a sequence of spaces $X_i$. Then the induced morphisms
\[
\begin{array}{c c c}
\mu_f:\bigotimes^m_{i=1}H^*_f(X_i)\to H^*_f(\prod^m_{i=1}X_i)
&\text{and}
&\bar{\mu}_f:\bigotimes^m_{i=1}\tilde{H}^*_f(X_i)\to\tilde{H}^*_f(\bigwedge^m_{i=1}X_i).
\end{array}
\]
are isomorphisms as algebras, and there is a commutative diagram
\[
\xymatrix{
\bigotimes^m_{i=1}\tilde{H}^*_f(X_i)\ar[r]^-{\bar{\mu}_f}\ar[d]	&\tilde{H}^*_f(\bigwedge^m_{i=1}X_i)\ar[d]^-{q^*_f}\\
\bigotimes^m_{i=1}H^*_f(X_i)\ar[r]^-{\mu_f}						&H^*_f(\prod^m_{i=1}X_i)
}\]
where $q^*_f$ is induced by the quotient map $q:\prod^m_{i=1}X_i\to\bigwedge^m_{i=1}X_i$.
\end{lemma}

\begin{proof}
It suffices to show the $m=2$ case. Let $(X, A)$ and $(Y, B)$ be pairs of relative CW-complexes and let $\pi_X:(X\times Y, A\times Y)\to(X, A)$ and $\pi_Y:(X\times Y, X\times B)\to(Y, B)$ be projections. By the generalized version of K\"{u}nneth Theorem \cite[Chapter XIII, Theorem 11.2]{massey}, the sequence
\[
0\to\bigoplus_{i+j=n}H^i(X, A)\otimes H^j(Y, B)\overset{\mu'}{\to}H^n(X\times Y, X\times B\cup A\times Y)\to T\to0
\]
where $T$ is a torsion term and $\mu'$ sends $u\otimes v\in H^i(X, A)\otimes H^j(Y, B)$ to $\pi^*_X(u)\cup\pi^*_Y(v)$, is split exact. By Lemma~\ref{lemma_tensor product free} $(H^*(X, A)\otimes H^*(Y, B))_f\cong H^*_f(X, A)\otimes H^*_f(Y, B)$ and by Lemma~\ref{lemma_free SES}
\[
\mu'_f:H^*_f(X, A)\otimes H^*_f(Y, B)\longrightarrow H^*_f(X\times Y, X\times B\cup A\times Y)
\]
is an isomorphism. Since $\mu'$ is multiplicative, so is $\mu'_f$. Letting $A$ and $B$ be the basepoints of $X$ and $Y$ or be the empty set gives the isomorphisms
\[
\begin{array}{c c c}
\mu_f:H^*_f(X)\otimes H^*_f(Y)\cong H^*_f(X\times Y)
&\text{and}
&\bar{\mu}_f:\tilde{H}^*_f(X)\otimes\tilde{H}^*_f(Y)\cong\tilde{H}^*_f(X\wedge Y).
\end{array}
\]
The asserted commutative diagram follows from commutative diagram
\[
\xymatrix{
\bigotimes^m_{i=1}\tilde{H}^*(X_i)\ar[r]^-{\bar{\mu}}\ar[d]	&\tilde{H}^*(\bigwedge^m_{i=1}X_i)\ar[d]^-{q^*}\\
\bigotimes^m_{i=1}H^*(X_i)\ar[r]^-{\mu}						&H^*(\prod^m_{i=1}X_i)
}\]
\end{proof}

\begin{prop}\label{lemma_free SR ring}
Let $\underline{X}=\{X_i\}^m_{i=1}$ be a sequence of spaces $X_i$, and let $ K$ be a simplicial complex on $[m]$. Then the inclusion $\imath:(\underline{X},\ast)^{K}\to\prod^m_{i=1}X_i$ induces a ring isomorphism
\[
H^*_f((\underline{X},\ast)^{K})\cong\left(\bigotimes^m_{i=1}H^*_f(X_i)\right)/I_{K}
\]
where $I_{K}$ is generated by $x_{j_1}\otimes\cdots\otimes x_{j_k}$ for $x_{j_i}\in\tilde{H}^*_f(X_{j_i})$ and $\{j_1,\ldots,j_k\}\notin K$.
\end{prop}

\begin{proof}
This proof modifies the proofs in~\cite{BBCG, BP}. Consider homotopy cofibration sequence
\[
(\underline{X},\ast)^{K}\overset{\imath}{\to}\prod^m_{i=1}X_i\overset{\jmath}{\to}C,
\]
where $C$ is the mapping cone of $\imath$ and $\jmath$ is the inclusion. Suspend it and obtain a diagram of homotopy cofibration sequences
\begin{equation}\label{ES_Suspend X^m decomposition}
\xymatrix{
\Sigma(\underline{X},\ast)^{K}\ar[r]^-{\Sigma\imath}\ar[d]^-{a}	&\Sigma\left(\prod^m_{i=1}X_i\right)\ar[r]^-{\Sigma\jmath}\ar[d]^-{b}	&\Sigma C\ar[d]^-{c}\\
\bigvee_{J\in K}\Sigma\underline{X}^{\wedge J}\ar[r]^-{\bar{\imath}}	&\bigvee_{J\in [m]}\Sigma\underline{X}^{\wedge J}\ar[r]^-{\bar{\jmath}}	&\bigvee_{J\notin K}\Sigma\underline{X}^{\wedge J}
}
\end{equation}
where $\underline{X}^{\wedge J}=X_{j_1}\wedge\cdots\wedge X_{j_k}$ for $J=\{j_1,\ldots,j_k\}$, $\bar{\imath}$ is the inclusion, $\bar{\jmath}$ is the pinch map, $a$ is a homotopy equivalence by~\cite[Theorem 2.21]{BBCG}, $b$ is a homotopy equivalence, and $c$ is an induced homotopy equivalence. Take cohomology and get the diagram
\[
\xymatrix{
0\ar[r]	&\bigoplus_{J\notin K}\tilde{H}^*(\underline{X}^{\wedge J})\ar[r]^-{\bar{\jmath}^*}\ar[d]^-{c^*}	&\bigoplus_{J\in[m]}\tilde{H}^*(\underline{X}^{\wedge J})\ar[r]^-{\bar{\imath}^*}\ar[d]^-{b^*}	&\bigoplus_{J\in K}\tilde{H}^*(\underline{X}^{\wedge J})\ar[r]\ar[d]^-{a^*}	&0\\
0\ar[r]	&\tilde{H}^*(C)\ar[r]^-{\jmath^*}	&\tilde{H}^*(\prod^m_{i=1}X_i)\ar[r]^-{\imath^*}	&\tilde{H}^*((\underline{X},\ast)^K)\ar[r]	&0
}
\]
The rows are split exact sequences. All vertical maps are isomorphisms. All maps are additive while $\imath^*$ is multiplicative. Apply Lemma~\ref{lemma_free SES} to the diagram and get
\begin{equation}\label{ES_SR ring proof torsionfree}
\xymatrix{
0\ar[r]	&\bigoplus_{J\notin K}\tilde{H}^*_f(\underline{X}^{\wedge J})\ar[r]^-{\bar{\jmath}^*_f}\ar[d]^-{c^*_f}	&\bigoplus_{J\in[m]}\tilde{H}^*_f(\underline{X}^{\wedge J})\ar[r]^-{\bar{\imath}^*_f}\ar[d]^-{b^*_f}	&\bigoplus_{J\in K}\tilde{H}^*_f(\underline{X}^{\wedge J})\ar[r]\ar[d]^-{a^*_f}	&0\\
0\ar[r]	&\tilde{H}^*_f(C)\ar[r]^-{\jmath^*_f}	&\tilde{H}^*_f(\prod^m_{i=1}X_i)\ar[r]^-{\imath^*_f}	&\tilde{H}^*_f((\underline{X},\ast)^K)\ar[r]	&0
}
\end{equation}
By Lemma~\ref{lemma_free Kunneth}, $H^*_f(\prod^m_{i=1}X_i)\cong\bigotimes^m_{i=1}H^*_f(X_i)$ so there is a ring isomorphism
\[
H^*_f((\underline{X},\ast)^{K})\cong\left(\bigotimes^m_{i=1}H^*_f(X_i)\right)/\text{ker}(\imath^*_f).
\]
Since the rows are split exact and the vertical maps are isomorphic in~(\ref{ES_SR ring proof torsionfree}), $\text{ker}(\imath^*_f)$ is generated by $x_{j_1}\otimes\cdots\otimes x_{j_k}$ for $x_{j_i}\in\tilde{H}^*_f(X_{j_i})$ and $\{j_1,\ldots,j_k\}\notin K$. Therefore $\text{ker}(\imath^*_f)=I_K$ and~$H^*_f((\underline{X},\ast)^{K})\cong\left(\bigotimes^m_{i=1}H^*_f(X_i)\right)/I_K$.
\end{proof}

Proposition~\ref{lemma_free SR ring} can be refined as follows. If the quotient map $H^*(X_i)\to H^*_f(X_i)$ has right inverse for all $i$, then so does $H^*((\underline{X},\ast)^K)\to H^*_f((\underline{X},\ast)^K)$. To formulate this, we introduce new definition.

\begin{dfn}
A graded algebra $\mathcal{A}$ is \emph{free split} if the quotient map $\pi:\mathcal{A}\to\mathcal{A}_f$ has a section as algebras. In other words, there is a ring morphism $s:\mathcal{A}_f\to\mathcal{A}$ making the following diagrams commute
\[
\begin{array}{c c}
\xymatrix{
\mathcal{A}_f\ar[r]^-{s}\ar@{=}[dr]	&\mathcal{A}\ar[d]^-{\pi}\\
&\mathcal{A}_f
}
&
\xymatrix{
\mathcal{A}_f\otimes\mathcal{A}_f\ar[d]^-{s\otimes s}\ar[r]^-{m_f}	&\mathcal{A}_f\ar[d]^-{s}\\
\mathcal{A}\otimes \mathcal{A}\ar[r]^-{m}							&\mathcal{A},
}
\end{array}
\]
where $m$ and $m_f$ are multiplications in $\mathcal{A}$ and $\mathcal{A}_f$. We call $s$ a \emph{free splitting} of $\mathcal{A}$.
\end{dfn}

In general, a free splitting of $\mathcal{A}$ is not unique. Any two free splittings $s_1$ and $s_2$ differ by a torsion element.

\begin{remark}
Not all cohomology rings of spaces are free split. Let $C$ be the mapping cone of the composite
\[
P^3(2)\overset{\rho}{\longrightarrow}S^3\overset{[\imath_1,\imath_2]}{\longrightarrow}S^2\vee S^2,
\]
where $P^3(2)$ is the mapping cone of degree map $2:S^2{\to}S^2$, $\rho$ is the quotient map and $[\imath_1,\imath_2]$ is the Whitehead product. Then $H^*(C)\cong\Z[a,b]/(a^2=b^2=2ab=0)$ where $|a|=|b|=2$, and it is not free split.
\end{remark}

\begin{lemma}\label{lemma_free split}
Under the conditions of Proposition~\ref{lemma_free SR ring}, if $H^*(X_i)$ is free split for all $i$, then~$H^*((\underline{X},\ast)^{K})$ is free split.
\end{lemma}

\begin{proof}
Use the notations in the proof of Proposition~\ref{lemma_free SR ring}. For $1\leq i\leq m$, let~\mbox{$s_i:H^*_f(X_i)\to H^*(X_i)$} be a free splitting and let $s$ be the composite
\[
s:\bigotimes^m_{i=1}H^*_f(X_i)\overset{\bigotimes^m_{i=1}s_i}{\longrightarrow}\bigotimes^m_{i=1}H^*(X_i)\overset{\mu}{\longrightarrow}H^*(\prod^m_{i=1}X_i).
\]
Then $s$ is a free splitting of $H^*(\prod^m_{i=1}X_i)$. As $\imath^*_f:H^*_f(\prod^m_{i=1}X_i)\to H^*_f((\underline{X},\ast)^K)$ is surjective, define $s':H^*_f((\underline{X},\ast)^K)\to H^*((\underline{X},\ast)^K)$ by the diagram
\[
\xymatrix{
\bigotimes^m_{i=1}H^*_f(X_i)\ar[r]^-{s}\ar[d]^-{\imath^*_f}	&H^*(\prod^m_{i=1}X_i)\ar[d]^-{\imath^*}\\
H^*_f((\underline{X},\ast)^{K})\ar[r]^-{s'}	&H^*((\underline{X},\ast)^{K})
}
\]
We need to show that $s'$ is well defined. For $x\in H^*_f(\underline{X},\ast)^K$, let $y,y'\in\bigotimes^m_{i=1}H^*_f(X_i)$ be two preimages of $x$. Then $y-y'\in\text{ker}(\imath^*_f)=I_K$. For $J\notin K$, $s$ sends $\tilde{H}^*_f(\underline{X})^{\otimes J}$ to~$\mu\left(\tilde{H}^*(\underline{X})^{\otimes J}\right)$ which is contained in $\text{ker}(\imath^*)$. So $\imath^*\circ s(y-y')=0$ and $s'$ is well defined. Since $s$, $\imath^*$ and $\imath^*_f$ are multiplicative, so is $s'$. Furthermore, $s'$ is right inverse to the quotient map~$H^*((\underline{X},\ast)^{K})\to H^*_f((\underline{X},\ast)^{K})$. So $s'$ is a free splitting.
\end{proof}

\section{Realization of graded monomial ideal rings}

We follow the idea of~\cite{BBCG} and prove Theorem~\ref{thm_main thm} in several steps. In Section 4.1 we use Proposition~\ref{lemma_free SR ring} to prove the special case where the ideal $I$ of $A$ is square-free. In Sections~4.2 and~4.3 we construct a fibration sequence inspired by algebraic polarization method and show that the fiber $X_A$ is a realization modulo torsion of $A$. More precisely, we apply the Eilenberg-Moore spectral sequence defined in Section~\ref{section_EMSS} to calculate $H^*_f(X_A)$ and give the~$E_{\infty}$-page by the end of this section. The extension problem is long and complicated and will be discussed in Section 5.

\subsection{Quotient rings of square-free ideals}
Let $P=\Z[x_1,\ldots,x_m]\otimes\Lambda[y_1,\ldots,y_n]$ be a graded polynomial ring where $x_i$'s have arbitrary positive even degrees and $y_j$'s have arbitrary positive odd degrees, and let $I=(M_1,\ldots,M_r)$ be an ideal generated by monomials
\[
M_j=x_1^{a_{1j}}\cdots x_m^{a_{mj}}\otimes y_1^{b_{1j}}\cdots y_n^{b_{nj}},
\]
where $a_{ij}$'s are non-negative integers and $b_{ij}$'s are either $0$ or $1$. Then $A=P/I$ is a \emph{graded monomial ideal ring}. We say that $I$ is \emph{square-free} if $M_j$'s are square-free monomials, that is all $a_{ij}$'s are either 0 or 1.

In the following let
\begin{itemize}
\item
$\{i_1,\ldots,i_k\}+\{j_1,\ldots,j_l\}=\{i_1,\ldots,i_k,j_1+m,\ldots,j_l+m\}$ for $\{i_1,\ldots,i_k\}\subset[m]$ and $\{j_1,\ldots,j_l\}\subset[n]$ and
\item
$\underline{X}+\underline{Y}=\{X_1,\ldots,X_m,Y_1\ldots,Y_n\}$ for $\underline{X}=\{X_i\}^m_{i=1}$ and $\underline{Y}=\{Y_j\}^n_{j=1}$ sequences of spaces.
\end{itemize}
Given a square-free ideal $I$ of $A$, take $K$ to be a simplicial complex on $[m+n]$ by removing faces faces $\{i_1,\ldots,i_k\}+\{j_1,\ldots,j_l\}$ whenever $x_{i_1}\cdots x_{i_k}\otimes y_{j_1}\cdots y_{j_l}\in I$. Then $I$ is the generalized Stanley-Reisner ideal of $ K$.

\begin{lemma}\label{lemma_square free realization}
Let $\underline{X}=\{K(\Z,|x_i|)\}^m_{i=1}$ and $\underline{Y}=\{S^{|y_j|}\}^n_{j=1}$ and let $K$ be the simplicial complex defined as above. Then there is a ring isomorphism $H^*_f((\underline{X}+\underline{Y} , \ast)^{K})\cong A$. Furthermore,~$H^*((\underline{X}+\underline{Y} , \ast)^{K})$ is free split.
\end{lemma}

\begin{proof}
Since $H^*_f(X_i)\cong\Z[x_i]$ and $H^*(Y_j)\cong\Lambda[y_j]$, the first part follows from Proposition~\ref{lemma_free SR ring}.

For the second part, it suffices to show that $H^*(X_i)$ and $H^*(Y_j)$ are free split by Lemma~\ref{lemma_free split}. For $1\leq j\leq n$, $H^*(Y_j)$ is free and hence free split. For $1\leq i\leq m$, let $x'_i$ be a generator of~$H^{|x_i|}(X_i)\cong\Z$. Then inclusion $\imath:\Z\langle{x'_i}\rangle\to H^*(X_i)$ extends to a ring morphism
\[
s:\Z[x'_i]\cong H^*_f(X_i)\to H^*(X_i).
\]
Let $\pi:H^*(X_i)\to H^*_f(X_i)$ be the quotient map. Since $\pi\circ\imath$ sends $x'_i$ to itself, by universal property $\pi\circ s$ is the identity map. So $s$ is a free splitting of $H^*(X_i)$.
\end{proof}

\subsection{Polarization of graded monomial ideal rings}

Now drop the square-free assumption on $I=(x_1^{a_{1j}}\cdots x_m^{a_{mj}}\otimes y_1^{b_{1j}}\cdots y_n^{b_{nj}}| 1\leq j\leq r)$ and suppose some $a_{ij}$'s are greater than~1. Following ideas from~\cite{BBCG} and~\cite{trevisan}, we use polarization to reduce the realization problem of $A$ to the special case when $I$ is square-free.

For $1\leq i\leq m$, let $a_i=\text{max}\{a_{i1},\ldots,a_{ir}\}$ be the largest index of $x_i$ among $M_j$'s, and let
\[
\Omega=\{(i,j)\in\N\times\N|1\leq i\leq m,1\leq j\leq a_i\}
\]
where $(i,j)\in\Omega$ are ordered in left lexicographical order, that is $(i,j)<(i',j')$ if $i<i'$, or if~$i=i'$ and $j<j'$. Let
\begin{eqnarray*}
P'
&=&\Z[x_{ij}|(i,j)\in\Omega]\otimes\Lambda[y_1,\ldots,y_n]\\
&=&\Z[x_{11},\ldots,x_{1a_1},x_{21},\ldots,x_{2a_2},\ldots,x_{m1},\ldots,x_{ma_m}]\otimes\Lambda[y_1,\ldots,y_n],
\end{eqnarray*}
be a graded polynomial ring where $|x_{ij}|=|x_i|$, let
\[
M'_j=(x_{11}x_{12}\cdots x_{1a_{1j}})(x_{21}x_{22}\cdots x_{2a_{2j}})\cdots(x_{m1}x_{m2}\cdots x_{ma_{mj}})\otimes(y^{b_{1j}}_1\cdots y^{b_{nj}}_n)
\]
and let $I'=(M'_1,\ldots,M'_r)$. Then $I'$ is square-free and $A'=P'/I'$ is called the \emph{polarization} of $A$.

Conversely, we can reverse the polarization process and obtain $A$ back from $A'$. Let
\[
\bar{\Omega}=\{(i,j)\in\N\times\N|1\leq i\leq m,2\leq j\leq a_i\}
\]
where $(i,j)\in\bar{\Omega}$ are ordered in left lexicographical order, and let $W$ be a graded polynomial ring
\begin{eqnarray*}
W
&=&\Z[w_{ij}|(i,j)\in\bar{\Omega}]\\
&=&\Z[w_{12},\ldots,w_{1a_1},w_{22},\ldots,w_{2a_2},\ldots,w_{m2},\ldots,w_{ma_m}]
\end{eqnarray*}
where $|w_{ij}|=|x_i|$. Define a ring morphism $\delta:W\to P'$ by $\delta(w_{ij})=x_{ij}-x_{i1}$ and make $P'$ a~$W$-module via $\delta$. Then $A'$ is a $W$-module and $A\cong A'/\bar{W}A'$, where $\bar{W}=\{W^i\}_{i>0}$.

\begin{lemma}\label{lemma_(u_i) is regular and free}
Let $A'$ be a square-free graded monomial ideal ring and let $W$ and $\delta$ be defined as above. Then $A'$ is a free $W$-module.
\end{lemma}

\begin{proof}
Since $A'/\bar{W}A'$ is a free $\Z$-module, by Lemma~\ref{lemma_regular freeness} it suffices to show that $\{w_{ij}\}_{(i,j)\in\bar{\Omega}}$ is a $A'$-regular sequence. Denote $N=|\bar{\Omega}|=\sum^m_{i=1}a_i-m$. For $1\leq k\leq N$, let $(i_k,j_k)\in\bar{\Omega}$ be the $k^{\text{th}}$ pair under lexicographical order and let $I_k=(w_{12},w_{13},\ldots,w_{i_kj_k})$. We need to show that multiplication $w_{i_{k+1}j_{k+1}}:A'/I_kA'\to A'/I_kA'$ is injective.

Observe that $A'/I_kA'=\tilde{P}/\tilde{I}$, where
\[
\tilde{P}=\Z[x_{11},x_{21},\ldots,x_{m1},x_{i_{k+1}j_{k+1}},x_{i_{k+2}j_{k+2}},\ldots x_{i_{N}j_{N}}]\otimes\Lambda[y_1,\ldots,y_n]
\]
and $\tilde{I}=(\tilde{M}_1,\ldots,\tilde{M}_r)$ is generated by monomials $\tilde{M}_j$ obtained by identifying $x_{ij}$ with $x_{i1}$ in~$M'_j$ for $(i,j)\leq (i_k,j_k)$. Suppose there is a polynomial $p\in\tilde{P}$ such that
\[
(x_{i_{k+1}j_{k+1}}-x_{i_{k+1}1})\cdot p\in\tilde{I}.
\]
We can use the combinatoric argument of~\cite[Page 31]{froberg} to show $p\in\tilde{I}$. Here is an outline of the argument. Write $p=\sum_{\alpha}p_{\alpha}$ as a sum of monomials $p_{\alpha}$. For each monomial $p_{\alpha}$, it can be shown that $x_{i_{k+1}j_{k+1}}p_{\alpha}$ and $x_{i_{k+1}1}p_{\alpha}$ are in $\tilde{I}$. Counting the indices of variables implies~$p_{\alpha}\in\tilde{I}$. So $p$ is in $\tilde{I}$ and multiplication $w_{i_{k+1}j_{k+1}}:A'/I_kA'\to A'/I_kA'$ is injective. Therefore $\{w_{ij}\}_{(i,j)\in\bar{\Omega}}$ is a regular sequence and $A'$ is a free $W$-module.
\end{proof}

\subsection{Constructing a realization modulo torsion $X_A$}
Let $A'=P'/I'$ be the polarization of $A$ and let~$K$ be a simplicial complex on $(\sum^m_{i=1}a_i+n)$ vertices such that $I'$ is the generalized Stanley-Reisner ideal of $K$. Construct a polyhedral product to realize $A'$. Take
\begin{eqnarray*}
\underline{X}
&=&\{X_{ij}=K(\Z,|x_i|)|(i,j)\in\Omega\}\\
&=&\{\underbrace{K(\Z,|x_1|),\ldots,K(\Z,|x_1|)}_{a_1},\underbrace{K(\Z,|x_2|),\ldots,K(\Z,|x_2|)}_{a_2}, \ldots, \underbrace{K(\Z,|x_m|),\ldots,K(\Z,|x_m|)}_{a_m}\}
\end{eqnarray*}
and
\[
\underline{Y}=\{Y_k=S^{|y_k|}|1\leq k\leq n\}=\{S^{|y_1|},S^{|y_2|},\ldots,S^{|y_n|}\}.
\]
By Lemma~\ref{lemma_square free realization} $H^*_f((\underline{X}+\underline{Y}, \ast)^{K})$ is isomorphic to $A'$.

For $1\leq i\leq m$, define $\delta_i:\prod^{a_i}_{j=1}X_{ij}\to\prod^{a_i}_{j=2}X_{ij}$ by $\delta_i(u_1,\ldots,u_{a_i})=(u_2\cdot u_1^{-1},\ldots,u_{a_i}\cdot u_1^{-1})$, and define $\delta:(\underline{X}+\underline{Y},\ast)^{K}\to\prod_{(i,j)\in\bar{\Omega}}X_{ij}$ to be the composite
\[
\delta:(\underline{X}+\underline{Y},\ast)^{K}\hookrightarrow\prod_{(i,j)\in\Omega}X_{ij}\times\prod^n_{k=1}Y_k\overset{\text{proj}}{\longrightarrow}\prod_{(i,j)\in\Omega}X_{ij}\overset{\prod^m_{i=1}\delta_i}{\longrightarrow}\prod_{(i,j)\in\bar{\Omega}}X_{ij}.
\]
As $\delta$ is a fibration, take $X_A$ to be its fiber. We claim $H^*_f(X_A)\cong A$.

\begin{notation}\label{notation_eilenberg-moore for X_A}
Let $\{E^{*,*}_r\}^{\infty}_{r=0}$ be the Eilenberg-Moore spectral sequence defined in Section~\ref{section_EMSS} on the fibration sequence
\begin{equation}\label{fib_polarization}
X_A\longrightarrow(\underline{X}+\underline{Y},\ast)^{K}\overset{\delta}{\longrightarrow}\prod_{(ij)\in\bar{\Omega}}X_{ij},
\end{equation}
where $H^*((\underline{X}+\underline{Y},\ast)^{K})$ is an $H^*(\prod_{(ij)\in\bar{\Omega}}X_{ij})$-module via $\delta^*$. 
\end{notation}

\begin{lemma}\label{lemma_E_infty}
For $E_{\infty}$-page, $(E^{0,q}_{\infty})_f\cong A^q$ as modules and $(E^{-p,q}_{\infty})_f=0$ for $p\neq0$.
\end{lemma}

\begin{proof}
The $E_2$-page is given by $E^{-p,*}_2=\text{Tor}^{-p,*}_{H^*(\prod_{(ij)\in\bar{\Omega}}X_{ij})}(H^*((\underline{X}+\underline{Y},\ast)^{K}),\Z)$. By Lemma~\ref{lemma_free E2_alg}, there is a monomorphism
\[
\pi':(E^{-p,*}_2)_f\longrightarrow\left(\text{Tor}^{-p,*}_{H^*_f(\prod_{(ij)\in\bar{\Omega}}X_{ij})}(H^*_f((\underline{X}+\underline{Y},\ast)^{K}),\Z)\right)_f,
\]
which is an isomorphism for $p=0$. By Lemmas~\ref{lemma_free Kunneth} and~\ref{lemma_free SR ring}, $H^*_f((\underline{X}+\underline{Y},\ast)^{K})\cong A'$ and
\[
H^*_f(\prod_{(ij)\in\bar{\Omega}}X_{ij})
\cong\Z[w_{12},\ldots,w_{1a_1},w_{22},\ldots,w_{2a_2},\ldots,w_{m2},\ldots,w_{ma_m}].
\]
Denote $H^*_f(\prod_{(ij)\in\bar{\Omega}}X_{ij})$ by $W$. So $A'$ is a $W$-module via $\delta^*$. By Lemma~\ref{lemma_(u_i) is regular and free} $A'$ is a free~$W$-module, so
\[
\text{Tor}^{-p,q}_{W}(A',\Z)\cong\begin{cases}
A^q	&p=0\\
0	&\text{otherwise}.
\end{cases}
\]
It follows that $(E^{-p,q}_2)_f$ is $A^q$ for $p=0$ and is zero otherwise.

Suppose $(E^{-p,q}_r)_f$ is $A^q$ for $p=0$ and is zero otherwise. Since $(E^{-p,*}_r)_f$ is concentrated in the column $p=0$, any differentials $d_r$ in and out of torsion-free elements is trivial. So we have $\text{ker}(d_r)_f=(E^{-p,q}_r)_f$ and $\text{Im}(d_r)_f=0$. By Lemma~\ref{lemma_free SES}, $(E^{-p,q}_{r+1})_f\cong(E^{-p,q}_r)_f$. Therefore~$(E^{-p,q}_{\infty})_f$ is isomorphic to $A^q$ for $p=0$ and is zero otherwise.
\end{proof}

\begin{lemma}\label{lemma_H(XS)=S as mod}
There is an additive isomorphism $H^q_f(X_A)\cong A^q$.
\end{lemma}

\begin{proof}
Since Eilenberg-Moore spectral sequence strongly converges to $H^*(X_A)$, for any fixed~$q$ there is a decreasing filtration $\{\mathscr{F}^{-p}\}$ of $H^q(X_A)$ such that
\[
\begin{array}{c c c}
\mathscr{F}^{-\infty}=H^{q}(X_A),
&\mathscr{F}^{1}=0,
&E^{-p,p+q}_{\infty}\cong\mathscr{F}^{-p}/\mathscr{F}^{-p+1}.
\end{array}
\]
By Lemma~\ref{lemma_free SES} $(E^{-p,p+q}_{\infty})_f\cong\left((\mathscr{F}^{-p})_f/(\mathscr{F}^{-p+1})_f\right)_f$. By Lemma~\ref{lemma_E_infty} $(E^{-p,p+q}_{\infty})_f$ is zero unless~$p=0$, so $H^q_f(X_A)\cong(E^{0,q}_{\infty})_f\cong A^q$ as modules. 
\end{proof}

Before going to the extension problem of the $E_{\infty}$-page, we consider the special case where all of the even degree generators of $A$ are in degree 2. The following theorem refines Lemma~\ref{lemma_H(XS)=S as mod} and shows that $H^*(X_A)\cong A$ as algebras without modding out the cohomology by torsion. This generalizes the results of Bahri-Bendersky-Cohen-Gitler~\cite[Theorem 2.2]{BBCG2} and Trevisan~\cite[Theorem 3.6]{trevisan}.

\begin{thm}\label{cor_deg 2 and odd deg}
Let $A$ be a graded monomial ideal ring where its generators have either degree~2 or arbitrary positive odd degrees. Then $H^*(X_A)\cong A$ as algebras.
\end{thm}

\begin{proof}
The $E_2$-page is given by $E^{-p,*}_2=\text{Tor}^{-p,*}_{H^*(\prod_{(ij)\in\bar{\Omega}}X_{ij})}(H^*((\underline{X}+\underline{Y},\ast)^{K}),\Z)$. By hypothesis, $X_{ij}=\C\PP^{\infty}$ for $(i,j)\in\Omega$ and $H^*(\prod_{(ij)\in\bar{\Omega}}X_{ij})$ and $H^*((\underline{X}+\underline{Y},\ast)^{K}$ are free. Following the argument in the proof of Lemma~\ref{lemma_E_infty}, $E^{-p,q}_2$ is $A^q$ for $p=0$ and is zero otherwise. Since the $E_2$-page is concentrated in the column $p=0$, the spectral sequence collapses and~$H^*(X_A)\cong A$ as modules.

Let $\phi:X_A\to(\underline{X}+\underline{Y},\ast)^K$ be the fiber inclusion. Lemma~\ref{lemma_edge homomorphism} implies the commutative diagram
\[
\xymatrix{
A'\ar[r]^-{\cong}\ar[d]^-{e}	&H^*((\underline{X}+\underline{Y},\ast)^K)\ar[d]^-{\phi^*}\\
E^{*,*}_2\ar[r]^-{\cong}		&H^*(X_A)
}
\]
where $e$ is surjective. Since $\phi^*$ is surjective and multiplicative and its kernel is $W$, $H^*(X_A)\cong A'/W\cong A$ as algebras.
\end{proof}

\section{The extension problem}
In this section we continue using Notation~\ref{notation_eilenberg-moore for X_A} in Section 4. Lemma~\ref{lemma_H(XS)=S as mod} shows that $H^*_f(X_A)$ and $A$ are free $\Z$-modules of same rank at each degree. We claim that they are isomorphic as algebras. The idea is to construct a space $Z_A$ related to $X_A$ such that $H^*(Z_A)$ is free and computable. Then we define a map $g_A:Z_A\to X_A$ and compare $H^*(X_A)$ with $H^*(Z_A)$ via~$g^*_A$.

\subsubsection*{Construction of $Z_A$}
For $1\leq i\leq m$ let $|x_i|=2c_i$, and let
\begin{eqnarray*}
\underline{Z}
&=&\{Z_{ij}=(\C\PP^{\infty})^{c_i}|(i,j)\in\Omega\}\\
&=&\{\underbrace{(\C\PP^{\infty})^{c_1},\ldots,(\C\PP^{\infty})^{c_1}}_{a_1},\underbrace{(\C\PP^{\infty})^{c_2},\ldots,(\C\PP^{\infty})^{c_2}}_{a_2},\ldots, \underbrace{(\C\PP^{\infty})^{c_m},\ldots,(\C\PP^{\infty})^{c_m}}_{a_m}\}
\end{eqnarray*}
and construct polyhedral product $(\underline{Z}+\underline{Y},\ast)^{K}$. Fix a generator $z$ of $H^2(\C\PP^{\infty})$. For~\mbox{$(i,j)\in\Omega$} and $1\leq k\leq c_i$, let $\pi_{ijk}:Z_{ij}\to\C\PP^{\infty}$ be the projection onto the $k^{\text{th}}$ copy of $\C\PP^{\infty}$ and let~$z_{ijk}=\pi^*_{ijk}(z)$. By Theorem~\ref{prop_BBCG} we have
\[
H^*((\underline{Z}+\underline{Y},\ast)^{K})\cong Q'/L',
\]
where $Q'=\Z[z_{ijk}|(i,j)\in\Omega,1\leq k\leq c_i]\otimes\Lambda[y_1,\ldots,y_n]$ and $L'$ is the ideal generated by monomials
\[
z_{i_1j_1k_1}\cdots z_{i_tj_tk_t}\otimes y_{l_1}\cdots y_{l_{\tau}}
\]
for $\{j_1+\sum^{i_1-1}_{s=1}a_s,\ldots,j_t+\sum^{i_t-1}_{s=1}a_s\}+\{l_1,\ldots,l_{\tau}\}\notin K$. For $1\leq i\leq m$, define
\[
\tilde{\delta}_i:\prod^{a_i}_{j=1}Z_{ij}\to\prod^{a_i}_{j=2}Z_{ij},\quad
\tilde{\delta}_i(u_1,\ldots,u_{a_i})=(u_2\cdot u_1^{-1},\ldots,u_{a_i}\cdot u_1^{-1}),
\]
and define $\tilde{\delta}:(\underline{Z}+\underline{Y},\ast)^{K}\to\prod_{(i,j)\in\bar{\Omega}}Z_{ij}$ to be the composite
\[
\tilde{\delta}:(\underline{Z}+\underline{Y},\ast)^{K}\hookrightarrow\prod_{(i,j)\in\Omega}Z_{ij}\times\prod^n_{k=1}Y_k\overset{\text{proj}}{\longrightarrow}\prod_{(i,j)\in\Omega}Z_{ij}\overset{\prod^m_{i=1}\tilde{\delta}_i}{\longrightarrow}\prod_{(i,j)\in\bar{\Omega}}Z_{ij}.
\]

\begin{lemma}\label{lemma_Z^K cohmlgy}
Let $Z_A$ be the fiber of $\delta'$. Then $H^*(Z_A)\cong Q/L$, where
\[
Q=\Z[z_{ik}|1\leq i\leq m, 1\leq k\leq c_i]\otimes\Lambda[y_1,\ldots,y_n]
\]
with $|z_{ik}|=2$ and $L$ is generated by monomials $z_{i_1k_1}\cdots z_{i_Nk_N}\otimes y^{b_{1j}}_1\cdots y^{b_{nj}}_n$ satisfying
\[
\begin{array}{c c c c}
1\leq j\leq r,
&1\leq k_l\leq c_{i_l}
&\text{and}
&(i_1,\ldots,i_N)=(\underbrace{1,\ldots,1}_{a_{1j}},\underbrace{2,\ldots,2}_{a_{2j}},\ldots,\underbrace{m,\ldots,m}_{a_{mj}}).
\end{array}
\]
\end{lemma}

\begin{proof}
Apply Eilenberg-Moore spectral sequence to fibration sequence
\[
Z_A\longrightarrow(\underline{Z}+\underline{Y},\ast)^{K}\overset{\tilde{\delta}}{\longrightarrow}\prod_{(i,j)\in\bar{\Omega}}Z_{ij}.
\]
The $E_2$-page is given by $\tilde{E}^{-p,*}_2=\text{Tor}^{-p,*}_{H^*(\prod_{(i,j)\in\bar{\Omega}}Z_{ij})}(\Z, H^*((\underline{Z}+\underline{Y},\ast)^{K}))$. By K\"{u}nneth Theorem
\[
H^*(\prod_{(i,j)\in\bar{\Omega}}Z_{ij})
\cong\Z[v_{ijk}|(i,j)\in\bar{\Omega},1\leq k\leq c_i]
\]
where $|v_{ijk}|=2$. Denoted $H^*(\prod_{(i,j)\in\bar{\Omega}}Z_{ij})$ by $V$. By definition $\tilde{\delta}^*(v_{ijk})=z_{ijk}-z_{i1k}$. This gives an action of $V$ on $Q'$. By Lemma~\ref{lemma_(u_i) is regular and free} $Q'/L'$ is a free $V$-module, so
\[
\text{Tor}^{-p,*}_{V}(Q'/L',\Z)=\begin{cases}
(Q'/L')/(z_{ijk}-z_{ilk})	&p=0\\
0							&\text{otherwise.}
\end{cases}
\]
Modding out $(z_{ijk}-z_{ilk})$ identifies $z_{ijk}$ with $z_{ilk}$ in $Q'/L'$, so $(Q'/L')/(z_{ijk}-z_{ilk})\cong Q/L$. Since the $E_2$-page is concentrated in the column $p=0$, $H^*(Z_A)\cong Q/L$.

Lemma~\ref{lemma_edge homomorphism} implies a commutative diagram
\[
\xymatrix{
Q'\ar[r]^-{\cong}\ar[d]^-{e}	&H^*((\underline{Z}+\underline{Y})^K)\ar[d]^-{\phi^*}\\
E^{*,*}_2\ar[r]^-{\cong}		&H^*(Z_A)
}
\]
where $e$ is surjective and $\phi^*$ is induced by the fiber inclusion $\phi:Z_A\to(\underline{Z}+\underline{Y})^K$. This implies $\phi^*$ is surjective. Since $\phi^*$ is multiplicative, $H^*(Z_A)\cong Q/L$ as algebras.
\end{proof}


\subsubsection*{Construction of $g_A$}

Fix a generator $z\in H^2(\C\PP^{\infty})$. For $1\leq j\leq c_i$, let $\pi_j:(\C\PP^{\infty})^{c_i}\to\C\PP^{\infty}$ be the projection onto the $j^{\text{th}}$ copy of $\C\PP^{\infty}$ and let $z_{j}=\pi^*_j(z)$. For $1\leq i\leq m$, take a map~$g_i:(\C\PP^{\infty})^{c_i}\to K(\Z,2c_i)$ that represents the cocycle class $z_{1}\cdots z_{c_i}\in H^{2c_i}((\C\PP^{\infty})^{c_i})$. For $(i,j)\in\Omega$, let $g_{ij}:Z_{ij}\to X_{ij}$ be $g_i$, and for $1\leq k\leq n$, let $h_k:Y_k\to Y_k$ be the identity map. Then $\{g_{ij},h_k|(i,j)\in\Omega,1\leq k\leq n\}$ induces a map $g_{K}:(\underline{Z}+\underline{Y},\ast)^{K}\to(\underline{X}+\underline{Y},\ast)^{K}$ by the functoriality of polyhedral products.

\begin{lemma}\label{lemma_g_A send generators}
Let $\{x_{ij},y_k|(i,j)\in\Omega,1\leq k\leq n\}$ be generators of $H^*_f((\underline{X}+\underline{Y},\ast)^K)\cong P'/I'$ and $\{z_{ijl},y'_k|(i,j)\in\Omega,1\leq l\leq c_i,1\leq k\leq n\}$ be generators of $H^*((\underline{Z}+\underline{Y},\ast)^K)\cong Q'/L'$. Then $(g^*_{K})_f(x_{ij})=\prod^{c_i}_{l=1}z_{ijl}$ and  $(g^*_{K})_f(y_k)=y'_k$.
\end{lemma}

\begin{proof}
There is a commutative diagram
\[
\xymatrix{
(\underline{Z}+\underline{Y},\ast)^{K}\ar[r]^-{\jmath}\ar[d]^-{g_{K}}	&\prod_{(i,j)\in\Omega}Z_{ij}\times\prod^n_{k=1}Y_k\ar[d]^-{g}\\
(\underline{X}+\underline{Y},\ast)^{K}\ar[r]^-{\imath}					&\prod_{(i,j)\in\Omega}X_{ij}\times\prod^n_{k=1}Y_k
}
\]
where $\imath$ and $\jmath$ are inclusions, $g=\prod_{(i,j)\in\Omega}g_{ij}\times\prod^n_{k=1}h_k$. Taking cohomology and modding out torsion elements, we obtain the commutative diagram
\[
\xymatrix{
P'\ar[r]^-{\imath^*_f}\ar[d]_-{g^*_f}	&P'/I'\ar[d]^-{(g^*_{K})_f}\\
Q'\ar[r]^-{\jmath^*}					&Q'/L'
}
\]
where $\imath^*_f$ and $\jmath^*$ are the quotient maps. Let $\tilde{x}_{ij},\tilde{y}_k\in P'$ and $\tilde{z}_{ijl},\tilde{y}'_k\in Q'$ be generators such that $\imath^*_f(\tilde{x}_{ij})=x_{ij}$, $\imath^*_f(\tilde{y}_k)=y_k$, $\jmath^*_f(\tilde{y}'_k)=y'_k$ and $\jmath^*_f(\tilde{z}_{ijl})=z_{ijl}$. By construction $g^*_f(\tilde{x}_{ij})=\prod^{c_i}_{l=1}\tilde{z}_{ijl}$ and $g^*_f(\tilde{y}_k)=\tilde{y}'_k$, so we have $(g^*_{K})_f(x_{ij})=\prod^{c_i}_{l=1}z_{ijl}$ and $(g^*_{K})_f(y_k)=y'_k$.
\end{proof}

\begin{lemma}\label{lemma_f_S map}
There is a map $g_A:Z_A\to X_A$ making the following diagram commutes
\[
\xymatrix{
Z_A\ar[d]^-{g_A}\ar[r]	&(\underline{Z}+\underline{Y},\ast)^{K}\ar[d]^-{g_{K}}\\
X_A\ar[r]				&(\underline{X}+\underline{Y}, \ast)^{K}
}\]
where the horizontal maps are the inclusion maps.
\end{lemma}

\begin{proof}
One may want to construct $g_A$ by showing the diagram
\[
\xymatrix{
(\underline{Z}+\underline{Y},\ast)^{K}\ar[r]^-{\tilde{\delta}}\ar[d]^-{g_{K}}	&\prod_{(i,j)\in\bar{\Omega}}Z_{ij}\ar[d]^-{\prod_{(i,j)\in\bar{\Omega}} g_{ij}}\\
(\underline{X}+\underline{Y},\ast)^{K}\ar[r]^-{\delta}					&\prod_{(i,j)\in\bar{\Omega}}X_{ij}
}
\]
commutes. However, as $(\prod_{(i,j)\in\bar{\Omega}}g_{ij})\circ\bar{\delta}$ and $\delta\circ g_{K}$ induce different morphisms on cohomology, the diagram cannot commute. Instead, we show that the composite
\[
Z_A\longrightarrow(\underline{Z}+\underline{Y},\ast)^{K}\overset{g_{K}}{\longrightarrow}(\underline{X}+\underline{Y},\ast)^{K}\overset{\delta}{\longrightarrow}\prod_{(i,j)\in\bar{\Omega}}X_{ij}
\]
is trivial. If so, there will exist a map $g_A:Z_A\to X_A$ as asserted since $X_A$ is the fiber of $\delta$.

By definition of $\bar{\delta}$ there is a commutative diagram
\[
\xymatrix{
(\underline{Z}+\underline{Y},\ast)^{K}\ar[r]^-{\tilde{\delta}}\ar[d]^-{\jmath}	&\prod_{(i,j)\in\bar{\Omega}}Z_{ij}\ar@{=}\\
\prod_{(i,j)\in\Omega}Z_{ij}\times\prod^n_{k=1}Y_k\ar[r]^-{\text{proj}}	&\prod_{(i,j)\in\Omega}Z_{ij}\ar[u]_-{\prod^m_{i=1}\tilde{\delta}_i}
}
\]
where $\jmath$ is the inclusion. Denote $(\prod^m_{i=1}\tilde{\delta}_i)\circ\text{proj}$ by $\tilde{\delta}'$ and extend the diagram to
\[
\xymatrix{
Z_A\ar[r]\ar[d]^-{e}	&(\underline{Z}+\underline{Y},\ast)^{K}\ar[r]^-{\tilde{\delta}}\ar[d]^-{\jmath}	&\prod_{(i,j)\in\bar{\Omega}}Z_{ij}\ar@{=}[d]\\
\prod^m_{i=1}(\C\PP^{\infty})^{c_i}\times\prod^n_{k=1}Y_k\ar[r]^-{\triangle'\times h}	&\prod_{(i,j)\in\Omega}Z_{ij}\times\prod^n_{k=1}Y_k\ar[r]^-{\tilde{\delta}'}	&\prod_{(i,j)\in\bar{\Omega}}Z_{ij}\\
}
\]
where $\triangle':\prod^m_{i=1}(\C\PP^{\infty})^{c_i}\to\prod^{a_i}_{j=1}Z_{ij}$ is the diagonal map, $h:\prod^n_{k=1}Y_k\to\prod^n_{k=1}Y_k$ is the identity map, and $e$ is an induced map. The top and the bottom row are fibration sequences. The left square fits into the following commutative diagram
\[
\xymatrix{
Z_A\ar[r]\ar[d]^-{e}	&(\underline{Z}+\underline{Y}, \ast)^{K}\ar[r]^-{g_{K}}\ar[d]^-{\jmath}	&(\underline{X}+\underline{Y},\ast)^{K}\ar[r]^-{\delta}\ar[d]^-{\imath}	&\prod_{\bar{\Omega}}X_{ij}\ar@{=}[d]\\
\prod^m_{i=1}(\C\PP^{\infty})^{c_i}\times\prod^n_{j=1}Y_j\ar[r]^-{\triangle'\times h}\ar[dr]_-{\prod g_i\times h}	&\prod_{\Omega}Z_{ij}\times\prod^n_{j=1}Y_j\ar[r]^-{\prod g_{ij}\times h}	&\prod_{\Omega}X_{ij}\times\prod^n_{k=1}Y_k\ar[r]^-{\delta'}	&\prod_{\bar{\Omega}}X_{ij}\\
	&\prod^m_{i=1}K(\Z,|x_i|)\times\prod^n_{j=1}Y_j\ar[ur]_-{\triangle\times h}	&	&
}
\]
where $\imath$ is the inclusion, $\triangle:\prod^m_{i=1}K(\Z,|x_i|)\to\prod^{a_i}_{j=1}X_{ij}$ is the diagonal map, and $\delta'$ is the composite
\[
\delta':\prod_{(i,j)\in\Omega}X_{ij}\times\prod^n_{k=1}Y_k\overset{\text{proj}}{\longrightarrow}\prod_{(i,j)\in\Omega}X_{ij}\overset{\prod^m_{i=1}\delta_i}{\longrightarrow}\prod_{(i,j)\in\bar{\Omega}}X_{ij}.
\]
The middle square is due to the functoriality of polyhedral products, the right square is due to the definition of $\delta$ and the bottom triangle is due to the naturality of diagonal maps.

The composite of maps from $Z_A$ to $\prod_{(i,j)\in\bar{\Omega}}X_{ij}$ round the bottom triangle is trivial, since
\[
\prod^m_{i=1}K(\Z,|x_i|)\times\prod^n_{k=1}Y_k\overset{\triangle\times h}{\longrightarrow}\prod_{(i,j)\in\Omega}X_{ij}\times\prod^n_{k=1}Y_k\overset{\delta'}{\longrightarrow}\prod_{(i,j)\in\bar{\Omega}}X_{ij}
\]
is a fibration sequence. So the composite in the top row is trivial and this induces a map~\mbox{$g_A:Z_A\to X_A$} as asserted.
\end{proof}

Since $g^*_K:H^*((\underline{X}+\underline{Y})^K)\to H^*((\underline{Z}+\underline{Y})^K)$ is multiplicative and $H^*(Z_A)$ is a quotient algebra of $H^*((\underline{Z}+\underline{Y})^K)$, we use $g_A$ to compare $H^*(X_A)$ and $H^*(Z_A)$ and show that $H^*_f(X_A)$ is a quotient algebra of $H^*_f((\underline{X}+\underline{Y})^K)$.

\begin{lemma}\label{lemma_extension}
Let $\phi:X_A\to((\underline{X}+\underline{Y},\ast)^{K})$ be the inclusion. Then the induced morphism
\[
\phi^*_f:H^*_f((\underline{X}+\underline{Y},\ast)^{K})\longrightarrow H^*_f(X_A)
\]
is surjective and $ker(\phi^*_f)$ is generated by $x_{ij}-x_{i1}$ for ${(i,j)\in\bar{\Omega}}$.
\end{lemma}

\begin{proof}
Fix a positive integer $q$ and let $\psi:Z_A\to(\underline{Z}+\underline{Y},\ast)^{K}$ be the inclusion. Consider commutative diagram
\[
\xymatrix{
	&H^q((\underline{X}+\underline{Y},\ast)^{K})\ar[r]^-{g^*_{K}}\ar[d]^-{\phi^*}\ar[dl]_-{e}	&H^q((\underline{Z}+\underline{Y},\ast)^{K})\ar[d]^-{\psi^*}\\
E^{0,q}_{\infty}\ar[r]^-{h}	&H^q(X_A)\ar[r]^-{g^*_A}	&H^q(Z_A)
}
\]
where $e$ is surjective and $h$ is injective. The left triangle commutes due to Lemma~\ref{lemma_edge homomorphism} and the right square commutes due to Lemma~\ref{lemma_f_S map}. Mod out torsion elements and take a generator
\[
x_{i_1j_1}\cdots x_{i_sj_s}\otimes y_{l_1}\cdots y_{l_t}\in H^q_f((\underline{X}+\underline{Y},\ast)^{K}).
\]
By Lemma~\ref{lemma_g_A send generators} and the above diagram we have
\begin{eqnarray*}
(g^*_A\circ h\circ e)_f(x_{i_1j_1}\cdots x_{i_sj_s}\otimes y_{l_1}\cdots y_{l_t})&=&(\psi^*\circ g^*_{K})_f(x_{i_1j_1}\cdots x_{i_sj_s}\otimes y_{l_1}\cdots y_{l_t})\\
(g^*_A\circ h)_f(x_{i_1}\cdots x_{i_s}\otimes y_{l_1}\cdots y_{l_t})&=&\left(\prod^{s}_{u=1}\prod^{c_{i_u}}_{k=1}z_{i_uj_uk}\right)\otimes y_{l_1}\cdots y_{l_t}.
\end{eqnarray*}
Since $x_{i_1}\cdots x_{i_s}\otimes y_{l_1}\cdots y_{l_t}$ and $\left(\prod^{s}_{u=1}\prod^{c_{i_u}}_{k=1}z_{i_uj_uk}\right)\otimes y_{l_1}\cdots y_{l_t}$ are generators, $(g_A\circ h)^*_f$ is the inclusion of a direct summand into $H^q_f(Z_A)$. By Lemma~\ref{lemma_E_infty} $(E^{0,q}_{\infty})_f$ and $H^q_f(X_A)$ are free modules of same rank, so $h_f$ is an isomorphism. Since $e_f$ is a surjection, so is $\phi^*_f$.

For the second part of the lemma, suppose there is a polynomial $p\in\text{ker}(\phi^*_f)$ not contained in $(x_{ij}-x_{i1})_{(i,j)\in\bar{\Omega}}$. Since $\phi^*_f$ is a degree $0$ morphism, we assume $p=\sum_{\alpha}p_{\alpha}$ is a sum of monomials $p_{\alpha}$ of some fixed degree $q$. Then $p_{\alpha}$'s are linear dependent. So the rank of $H^{q}_f(X_A)$ is less than the rank of $A^{q}$, contradicting to Lemma~\ref{lemma_E_infty}. So $\text{ker}(\phi^*_f)=(x_{ij}-x_{i1})_{(i,j)\in\bar{\Omega}}$.
\end{proof}

Next we restate the Main Theorem (Theorem~\ref{thm_main thm}) and prove it.

\begin{thm}\label{thm_main new}
Let $A$ be a graded monomial ideal ring. Then there exists a space $X_A$ such that $H^*_f(X_A)$ is ring isomorphic to $A$. Moreover, $H^*(X_A)$ is free split.
\end{thm}

\begin{proof}
For the first part of the lemma, the ring isomorphism $H^*_f(X_A)\cong A$ follows from Lemma~\ref{lemma_extension}.

In Lemma~\ref{lemma_square free realization} we construct a free splitting $s_{K}:H^*_f(\underline{X}+\underline{Y},\ast)^{K}\to H^*(\underline{X}+\underline{Y},\ast)^{K}$ out of free splittings $s_{ij}:H^*_f(X_{ij})\to H^*(X_{ij})$ and the identity maps on $H^*(Y_k)$. Define a map~$s:H^*_f(X_A)\to H^*(X_A)$ by diagram
\[
\xymatrix{
H^*_f((\underline{X}+\underline{Y},\ast)^{K})\ar[r]^-{s_{K}}\ar[d]^-{\phi^*_f}	&H^*((\underline{X}+\underline{Y},\ast)^{K})\ar[d]^-{\phi^*}\\
H^*_f(X_A)\ar[r]^-{s}	&H^*(X_A)
}
\]
We need to show that $s$ is well defined. By Lemma~\ref{lemma_extension} $\phi^*_f$ is a surjection and $\text{ker}(\phi^*_f)$ is generated by polynomials $x_{ij}-x_{i1}$ for $(i,j)\in\bar{\Omega}$. It suffices to show $\phi^*\circ s_{K}(x_{ij}-x_{i1})=0$. Let $\tilde{x}_{ij}\in H^{2c_i}(X_{ij})$ and $\tilde{x}'_{ij}\in H^{2c_i}_f(X_{ij})$ be generators such that $s_{ij}(\tilde{x}'_{ij})=\tilde{x}_{ij}$. There is a string of equations
\begin{eqnarray*}
\phi^*\circ s_{K}(x_{ij}-x_{i1})
&=&\phi^*\circ\mu(s_{ij}(\tilde{x}'_{ij})-s_{i1}(\tilde{x}'_{i1}))\\
&=&\phi^*\circ\mu(\tilde{x}_{ij}-\tilde{x}_{i1})\\
&=&\phi^*\circ\delta^*\circ\mu(1\otimes\cdots\otimes \tilde{x}_{ij}\otimes\cdots\otimes1)\\
&=&0
\end{eqnarray*}
where the first line is due to the definition of $s_K$, the third line is due to the naturality of~$\mu$, and the last line is due to the fact that $\delta$ and $\phi$ are two consecutive maps in the fibration sequence $X_A\overset{\phi}{\to}(\underline{X}+\underline{Y},\ast)^K\overset{\delta}{\to}\prod_{(i,j)\in\bar{\Omega}}X_{ij}$. So $s$ is well defined.

Obviously $s$ is right inverse to the quotient map $H^*(X_A)\to H^*_f(X_A)$. Since $\phi^*_f$, $\phi^*$ and $s_{K}$ are multiplicative, so is $s$. Therefore $s$ is a free splitting.
\end{proof}

\section{An Example}

Now we illustrate how to construct $X_A$ for $A=\Z[x]\otimes\Lambda[y]/(x^2y)$, where $|x|=4$ and~$|y|=1$. First, polarize $A$ by introducing two new variables $x_{1}$ and $x_{2}$ of degree 4 and let
\[
A'=\Z[x_{1},x_{2}]\otimes\Lambda[y]/(x_{1}x_{2}y).
\]
Let $K$ be the boundary of a 2-simplex. Then $(x_{1}x_{2}y)$ is the Stanley-Reisner ideal of $K$. Take
\[
\begin{array}{c c}
\underline{X}=\{K(\Z,4),K(\Z,4)\},
&\underline{Y}=\{S^1\}
\end{array}
\]
and construct polyhedral product $(\underline{X}+\underline{Y},\ast)^{K}$. By Proposition~\ref{lemma_free SR ring}
\[
H^*_f((\underline{X}+\underline{Y},\ast)^{K})\cong\Z[x_{1},x_{2}]\otimes\Lambda[y]/(x_{1}x_{2}y).
\]
Define $\delta:(\underline{X}+\underline{Y},\ast)^K\to K(\Z,4)$ by $\delta_1(u_1,u_2,t)=u_2\cdot u_1^{-1}$, and define $X_A$ to be the fiber of $\delta$. By Theorem~\ref{thm_main new} $H^*_f(X_A)\cong A$.

Next, we construct $Z_A$ and $g_A$ to illustrate the proof of the extension problem. In this case, take $\underline{Z}=\{(\C\PP^{\infty})^2,(\C\PP^{\infty})^2\}$. Denote the first $(\C\PP^{\infty})^2$ by $Z_{1}$ and the second $(\C\PP^{\infty})^2$ by $Z_{2}$. Then $H^*(Z_{1})=\Z[z_{11},z_{12}]$ and $H^*(Z_{2})=\Z[z_{21},z_{22}]$, where $|z_{ij}|=2$ for $i,j\in\{1,2\}$, and
\[
H^*((\underline{Z}+\underline{Y},\ast)^K)\cong\Z[z_{11},z_{12},z_{21},z_{22}]\otimes\Lambda[y]/L'
\]
where $L'=(z_{11}z_{21}y, z_{11}z_{22}y, z_{12}z_{21}y, z_{12}z_{22}y)$. Define
\[
\tilde{\delta}:(\underline{Z}+\underline{Y},\ast)^K\to(\C\PP^{\infty})^2,\quad\tilde{\delta}(v_1,v_2,t)=v_2\cdot v_1^{-1},
\]
and define $Z_A$ to be the fiber of $\tilde{\delta}$. Then $H^*_f(Z_A)\cong\Z[z_{1},z_{2}]\otimes\Lambda[y]/L$, where $|z_{1}|=|z_{2}|=2$ and $L=(z^2_{1}y, z^2_{2}y, z_{1}z_{2}y)$.

For $i=\{1,2\}$, let $g_{i}:Z_{i}\to K(\Z,4)$ be a map representing $z_{i1}z_{i2}\in H^4(Z_{i})$, and let~\mbox{$h:S^1\to S^1$} be the identity map. Then $g_{1},g_{2}$ and $h$ induce $g_{K}:(\underline{Z}+\underline{Y},\ast)^{K}\to (\underline{X}+\underline{Y},\ast)^{K}$ such that $g^*_{K}(x_{i})=z_{i1}z_{i2}$ and $g^*_{K}(y)=y$. Lemma~\ref{lemma_f_S map} gives a map $g_A:Z_A\to X_A$ making the following diagram commute
\[
\xymatrix{
Z_A\ar[r]\ar[d]^-{g_A}	&(\underline{Z}+\underline{Y},\ast)^{K}\ar[d]^-{g_{K}}\\
X_A\ar[r]	&(\underline{X}+\underline{Y},\ast)^{K}.
}
\]

\section*{acknowledgment}
This work was done during the first author's PIMS postdoctoral Fellowship. Both authors gratefully acknowledge the support of Fields Institute and NSERC. The first author also thanks Hector Durham for looking for references.

\end{document}